\documentclass[12pt]{amsart}%
\usepackage{amsmath}
\usepackage{graphicx}
\usepackage[small,nohug,heads=littlevee]{diagrams}
\usepackage{amscd}
\usepackage{geometry}
\usepackage{amsfonts}
\usepackage{amssymb}%
\newtheorem{theorem}{Theorem}[section]
\theoremstyle{plain}

\newtheorem{corollary}[theorem]{Corollary}

\newtheorem{example}{Example}

\newtheorem{lemma}[theorem]{Lemma}

\newtheorem{proposition}[theorem]{Proposition}
\newtheorem{remark}{Remark}

\numberwithin{equation}{section}

\let\oldtocsection=\tocsection
\let\oldtocsubsection=\tocsubsection
\let\oldtocsubsubsection=\tocsubsubsection
\renewcommand{\tocsection}[2]{\hspace{0em}\oldtocsection{#1}{#2}}
\renewcommand{\tocsubsection}[2]{\hspace{2em}\oldtocsubsection{#1}{#2}}
\renewcommand{\tocsubsubsection}[2]{\hspace{4em}\oldtocsubsubsection{#1}{#2}}
\begin{document}
\title[Weak $\mathcal{Z}$-structures for some classes of groups]{Weak $\mathcal{Z}$-structures for some classes of groups}
\author{Craig R. Guilbault }
\address{Department of Mathematical Sciences, University of Wisconsin-Milwaukee,
Milwaukee, Wisconsin 53201}
\email{craigg@uwm.edu}
\thanks{Work on this project was aided by a Simons Foundation Collaboration Grant.}
\date{July 19, 2013}

\begin{abstract}
Motivated by the usefulness of boundaries in the study of $\delta$-hyperbolic
and CAT(0) groups, Bestvina introduced a general axiomatic approach to group
boundaries, with a goal of extending the theory and application of boundaries
to larger classes of groups. The key definition is that of a \textquotedblleft%
$\mathcal{Z}$-structure\textquotedblright\ on a group $G$. These $\mathcal{Z}%
$-structures, along with several variations, have been studied and existence
results obtained for\ a variety of new classes groups. Still, relatively
little is known about the general question of which groups admit any of the
various $\mathcal{Z}$-structures---aside from the (easy) fact that any such
$G$ must have type F, i.e., $G$ must admit a finite K($G,1$). In fact,
Bestvina has asked whether \emph{every }type F group admits a $\mathcal{Z}%
$-structure or at least a \textquotedblleft weak\textquotedblright%
\ $\mathcal{Z}$-structure.

In this paper we prove some general existence theorems for weak\ $\mathcal{Z}%
$-structures. The main results are as follows.\medskip

\noindent\textbf{Theorem A. }\emph{If }$G$ \emph{is an extension of \ a
nontrivial type F group by a nontrivial type F group, then }$G$\emph{ admits a
weak }$\mathcal{Z}$\emph{-structure.\smallskip}

\noindent\textbf{Theorem B.}\emph{ If }$G$\emph{ admits a finite K(}%
$G,1$\emph{) complex }$K$\emph{ such that the }$G$-\emph{action on
}$\widetilde{K}$\emph{ contains }$1\neq j\in G$\emph{ properly homotopic to
}$\operatorname{id}_{\widetilde{K}}$, \emph{then }$G$\emph{ admits a weak
}$\mathcal{Z}$\emph{-structure.\smallskip}

\noindent\noindent\textbf{Theorem C. }\emph{If }$G$\emph{ has type F and is
simply connected at infinity, then }$G$\emph{ admits a weak }$\mathcal{Z}%
$\emph{-structure.}\medskip

As a corollary of Theorem A or B, every type F group admits a weak
$\mathcal{Z}$-structure \textquotedblleft after
stabilization\textquotedblright; more precisely: if $H$ has type F, then
$H\times%
\mathbb{Z}
$ admits a weak $\mathcal{Z}$-structure. As another corollary of Theorem B,
every type F group with a nontrivial center admits a weak $\mathcal{Z}$-structure.

\end{abstract}
\maketitle
\tableofcontents

\section{Introduction\label{Section: Introduction}}

Several lines of investigation in geometric topology and geometric group
theory seek `nice' com\-pact\-ific\-at\-ions of contractible manifolds or
complexes (or ERs/ARs) on which a given group $G$ acts cocompactly as covering
transformations. Bestvina \cite{Be} has defined a $\mathcal{Z}$-structure and
a weak $\mathcal{Z}$-structure on a group $G$ as follows:\smallskip

\noindent A $\mathcal{Z}$\emph{-structure} on a group $G$ is a pair $\left(
\overline{X},Z\right)  $ of spaces satisfying:

\begin{enumerate}
\item $\overline{X}$ is a compact ER,

\item $Z$ is a $\mathcal{Z}$-set\footnote{The definition of $\mathcal{Z}$-set,
along with definitions of numeous other terms used in the introduction, can be
found in \S \ref{Section: Background}.} in $\overline{X}$,

\item $X=\overline{X}-Z$ admits a proper, free, cocompact action by $G$, and

\item (nullity condition) For any open cover $\mathcal{U}$ of $\overline{X}$,
and any compactum $K\subseteq X$, all but finitely many $G$-translates of $K$
lie in some element $U$ of $\mathcal{U}$.
\end{enumerate}

\noindent If only conditions 1)-3) are satisfied, $\left(  \overline
{X},Z\right)  $ is called a \emph{weak}\textbf{ }$\mathcal{Z}$%
\emph{-structure} on $G$.

An additional condition that can be added to conditions 1)-3), with or without
condition 4), is:

\begin{enumerate}
\item[(5)] The action of $G$ on $X$ extends to an action of $G$ on
$\overline{X}$.
\end{enumerate}

\noindent Farrell and Lafont \cite{FL} refer to a pair $\left(  \overline
{X},Z\right)  $ satisfying 1)-5) as an $E\mathcal{Z}$\emph{-structure}. Others
have considered pairs that satisfy 1)-3) and 5); we call those \emph{weak
}$E\mathcal{Z}$\emph{-structures.} Depending on the set of conditions
satisfied, $Z$ is referred to generically as a \emph{boundary} for $G$; or
more specifically as a $\mathcal{Z}$\emph{-boundary}, a \emph{weak
}$\mathcal{Z}$\emph{-boundary,} an $E\mathcal{Z}$\emph{-boundary}, or a
\emph{weak }$E\mathcal{Z}$\emph{-boundary}.

\begin{example}
A torsion-free group acting geometrically on a finite-dimensional CAT(0) space
$X$ admits an $E\mathcal{Z}$-structure---one compactifies $X$ by adding the
visual boundary. Bestvina and Mess \cite{BM} have shown that each torsion-free
word hyperbolic group $G$ admits an $E\mathcal{Z}$-structure $\left(
\overline{X},\partial G\right)  $, where $\partial G$ is the Gromov boundary
and $X$ is a Rips complex for $G$. Osajda and Przytycki have shown that
systolic groups admit $E\mathcal{Z}$-structures. \cite{Be} contains a
discussion of $\mathcal{Z}$-structures and weak $\mathcal{Z}$-structures on a
variety of other groups, not all of which satisfy condition 5).
\end{example}

\begin{remark}
\emph{Some authors (see \cite{Dr}) have extended the above definitions by
allowing non-free }$G$\emph{-actions (thus allowing for groups with torsion)
and by loosening the ER requirement on }$\overline{X}$\emph{ to that of AR,
i.e., allowing }$\overline{X}$\emph{ to be infinite-dimensional. Here we stay
with the original definitions, but note that some analogous results are
possible in the more general settings.}
\end{remark}

A group $G$ has \emph{type F} if it admits a finite K($G,1$) complex. The
following proposition narrows the field of candidates for admitting any sort
of $\mathcal{Z}$-structure to those groups of type F.

\begin{proposition}
If there exists a proper, free, cocompact $G$-action on an AR $Y$, then $G$
has type F.
\end{proposition}

\begin{proof}
The quotient $q:Y\rightarrow G\backslash Y$ is a covering projection, so
$G\backslash Y$ is aspherical and locally homeomorphic to $Y$. By the latter,
$G\backslash Y$ is a compact ANR, and thus (by Theorem
\ref{Theorem: West's Theorem}) homotopy equivalent to a finite complex. Any
such complex is a K($G,1$).
\end{proof}

In \cite{Be}, Bestvina asked the following pair of questions:\medskip

\noindent\textbf{Bestvina's Question. }\emph{Does every type F group admit a
}$\mathcal{Z}$\emph{-structure?\medskip}

\noindent\textbf{Weak Bestvina Question. }\emph{Does every type F group admit
a weak }$\mathcal{Z}$\emph{-structure?\medskip}

The Weak Bestvina Question was also posed by Geoghegan in \cite[p.425]{Ge2}.
Farrell and Lafont \cite{FL} have asked whether every type F group admits an
$E\mathcal{Z}$-structure, and the question of which groups admit weak
$E\mathcal{Z}$-structures appear in both \cite{BM} and \cite{Ge2}. Although
interesting special cases abound, a general solution to any of these questions
seems out of reach at this time.

As one would expect, the more conditions a $\mathcal{Z}$-structure or its
corresponding boundary satisfies, the greater the potential applications. For
example, Bestvina has shown that the topological dimension of a $\mathcal{Z}%
$-boundary is an invariant of the group---it is one less than the
cohomological dimension of $G$; this is not true for weak $\mathcal{Z}%
$-boundaries. But a weak $\mathcal{Z}$\emph{-}boundary carries significant
information about $G$. For example, the \v{C}ech cohomology of a weak
$\mathcal{Z}$\emph{-}boundary reveals the group cohomology of $G$ with $%
\mathbb{Z}
G$-coefficients, and the $\operatorname*{pro}$-homotopy groups of a weak
$\mathcal{Z}$-boundary are directly related to the corresponding end
invariants (such as simple connectivity at infinity) of $G$. A weak
$\mathcal{Z}$\emph{-}boundary, when it exists, is well-defined up to shape and
can provide a first step toward obtaining a stronger variety of $\mathcal{Z}%
$-structure on $G$. $E\mathcal{Z}$\emph{-} and weak\emph{ }$E\mathcal{Z}%
$-boundaries, when they exist, carry the potential for studying $G$ by
analyzing its action on the compactum Z. More about these topics can be found
in \cite{Ge1}, \cite{Ge2}, \cite{GM1}, \cite{Be}, \cite{FL} and \cite{Gu3}.

In this paper we prove the existence of weak $\mathcal{Z}$-structures for a
variety of groups. A notable special case provides a \textquotedblleft
stabilized solution\textquotedblright\ to the Weak Bestvina Question. It
asserts that, if $H$ has type F, then $H\times%
\mathbb{Z}
$ admits a weak $\mathcal{Z}$-structure. That result is an easy consequence of
either of the following more general theorems, to be proven here.

\begin{theorem}
\label{Theorem 1}If $G$ is an extension of a nontrivial type F group by a
nontrivial type F group, that is, if there exists a short exact sequence
$1\rightarrow N\rightarrow G\rightarrow Q\rightarrow1$ where $N$ and $Q$ are
nontrivial and type F, then $G$ admits a weak $\mathcal{Z}$-structure. More
generally, if a type F group $G$ is virtually an extension of a nontrivial
type F group by a nontrivial type F group, then $G$ admits a weak
$\mathcal{Z}$-structure.
\end{theorem}

\begin{theorem}
\label{Theorem 2}Suppose $G$ admits a finite K($G,1$) complex $K$, and the
corresponding $G$-action on the universal cover $\widetilde{K}$ contains a
$1\neq j\in G$ that is properly homotopic to $\operatorname{id}_{\widetilde
{K}}$. Then $G$ admits a weak $\mathcal{Z}$-structure.
\end{theorem}

\begin{remark}
\emph{For finite K(}$G,1$\emph{) complexes} $K$ \emph{and} $L$\emph{, or more
generally, compact aspherical ANRs with }$\pi_{1}\left(  K\right)  \cong
G\cong\pi_{1}\left(  L\right)  $\emph{, there is a }$G$\emph{-equivariant
proper homotopy equivalence }$\widetilde{f}:\widetilde{K}\rightarrow
\widetilde{L}$\emph{. If }$j\in G$\emph{ satisfies the hypothesis of Theorem
\ref{Theorem 2}, then so does }$\widetilde{f}\circ j$\emph{. Hence, the
existence of such a }$j$\emph{ can be viewed as a property of }$G$\emph{,
itself.}

\begin{example}
For a closed, orientable, aspherical $n$-manifold $M^{n}$ with $\widetilde
{M}^{n}\cong%
\mathbb{R}
^{n}$ (e.g., $M^{n}$ a Riemannian manifold of nonpositive sectional curvature)
every element of $\pi_{1}\left(  M^{n}\right)  $ satisfies the hypothesis of
Theorem \ref{Theorem 1}. On the other hand, for finitely generated free
groups, no elements do. Of course, weak $\mathcal{Z}$-structures for both of
these classes of groups are known for other reasons.
\end{example}
\end{remark}

\begin{corollary}
\label{Corollary: Stable weak Z-structures}If $H$ is type F, then $H\times%
\mathbb{Z}
$ admits a weak $\mathcal{Z}$-structure.
\end{corollary}

\begin{proof}
This corollary is immediate from Theorem \ref{Theorem 1}. Alternatively, it
may be obtained from Theorem \ref{Theorem 2}. Let $K$ be a finite K($H,1$) and
$H\times%
\mathbb{Z}
$ act diagonally on $\widetilde{K}\times%
\mathbb{R}
$. The nontrivial elements of $%
\mathbb{Z}
$ satisfy the hypotheses of that theorem.
\end{proof}

A more general application of Theorem \ref{Theorem 2} is the following.

\begin{corollary}
If $G$ is a type F group with a nontrivial center, then $G$ admits a weak
$\mathcal{Z}$-structure.
\end{corollary}

\begin{proof}
The point here is that, when $K$ is a finte K($G$,1), each nontrivial $j\in
Z\left(  G\right)  $ satisfies the hypothesis of Theorem \ref{Theorem 2}. A
complete proof of that fact can be deduced from \cite[Th.II.7]{Go}. We sketch
an alternative argument.

Let $K^{\ast}=K\cup A$ where $A$ is an arc with terminal point identified to a
vertex $p$ of $K$; let $p^{\ast}$ be the initial point of $A$. For arbitrary
$j\in\pi_{1}\left(  K,p\right)  $ define $f_{j}:\left(  K^{\ast},p^{\ast
}\right)  \rightarrow\left(  K^{\ast},p^{\ast}\right)  $ to be the identity on
$K$; to stretch the initial half of $A$ onto the image copy of $A$; and to
send the latter half of $A$ around a loop corresponding to $j$. The induced
homomorphism on $\pi_{1}\left(  K^{\ast},p^{\ast}\right)  $ is conjugation by
$j$. If $j\in Z\left(  G\right)  $ that homomorphism is the identity, so
$f_{j}$ is homotopic (rel $p^{\ast}$) to $\operatorname*{id}_{K^{\ast}}$.
Since $K^{\ast}$ is compact, that homotopy lifts to a proper homotopy
$\widetilde{F}:\widetilde{K^{\ast}}\times\left[  0,1\right]  \rightarrow
\widetilde{K^{\ast}}$. Collapse out the preimage of $A\times\left[
0,1\right]  $ in the domain and the preimage of $A$ in the range to get a
proper homotopy between $\operatorname*{id}_{\widetilde{K}}$ and the covering
translation corresponding to $j$.
\end{proof}

Theorems \ref{Theorem 1} and \ref{Theorem 2} will be obtained from a pair of
more general results, with hypotheses more topological than group-theoretic.

\begin{theorem}
\label{Theorem 3}Suppose $G$ admits a finite K($G,1$) complex $K$ with the
property that $\widetilde{K}$ is proper homotopy equivalent to a product
$X\times Y$ of noncompact ANRs, then $G$ admits a weak $\mathcal{Z}$-structure.
\end{theorem}

\begin{theorem}
\label{Theorem 4}Suppose $G$ admits a finite K($G,1$) complex $K$ for which
$\widetilde{K}$ is proper homotopy equivalent to an ANR $X$ that admits a
proper $%
\mathbb{Z}
$-action generated by a homeomorphism $h:X\rightarrow X$ that is properly
homotopic to $\operatorname*{id}_{X}$. Then $G$ admits a weak $\mathcal{Z}$-structure.
\end{theorem}

Note that neither the product structure in Theorem \ref{Theorem 3} nor the $%
\mathbb{Z}
$-action in Theorem \ref{Theorem 4} are required to have any relationship to
the $G$-action on $\widetilde{K}$.

A third variety of existence theorem for weak $\mathcal{Z}$-structures has, as
its primary hypothesis, a condition on the end behavior of $G.$

\begin{theorem}
\label{Theorem 5}If $G$ is type F, 1-ended, and has pro-monomorphic
fundamental group at infinity, then $G$ admits a weak $\mathcal{Z}$-structure.
\end{theorem}

\begin{corollary}
\label{Corollary: simply connected at infinity}If a type F group $G$ is simply
connected at infinity, then $G$ admits a weak $\mathcal{Z}$-structure.
\end{corollary}

Results found in \cite{Ja}, \cite{Mi}, \cite{Pr}, and \cite{CM} show that
simple connectivity at infinity is a common property for certain types of
group extensions. By applying those results, some interesting overlap can be
seen in the collections of groups covered by Corollary
\ref{Corollary: simply connected at infinity} and those covered by Theorems
\ref{Theorem 1} and \ref{Theorem 2}.\bigskip

In the next section, we introduce some terminology and review a number of
established results that are fundamental to later arguments. In
\S \ref{Section: Topological results}\ we prove a variety topological theorems
related to end properties of ANRs, complexes, and Hilbert cube manifolds. Most
importantly, we prove $\mathcal{Z}$-compactifiability for a variety of spaces.
Several results obtained there are more general than required for the
group-theoretic applications in this paper, and may be of independent
interest. In \S \ref{Section: Proofs of main theorems}, we prove the main
theorems stated above. In an appendix, we provide an alternative proof, based
on the theory of \textquotedblleft approximate fibrations\textquotedblright,
of a crucial lemma from \S \ref{Section: Proofs of main theorems}.

The author wishes to acknowledge Mike Mihalik and Ross Geoghegan for helpful
conversations that led to significant improments in this paper.

\section{Terminology and background\label{Section: Background}}

\subsection{Inverse sequences of groups}

Throughout this subsection all arrows denote homomorphisms, while arrows of
the type $\twoheadrightarrow$ or $\twoheadleftarrow$ denote surjections and
arrows of the type $\rightarrowtail$ and $\leftarrowtail$ denote injections.

Let
\[
G_{0}\overset{\lambda_{1}}{\longleftarrow}G_{1}\overset{\lambda_{2}%
}{\longleftarrow}G_{2}\overset{\lambda_{3}}{\longleftarrow}\cdots
\]
be an inverse sequence of groups. A \emph{subsequence} of $\left\{
G_{i},\lambda_{i}\right\}  $ is an inverse sequence of the form
\[
\begin{diagram}
G_{i_{0}} & \lTo^{\lambda_{i_{0}+1}\circ\cdots\circ\lambda_{i_{1}}
} & G_{i_{1}} & \lTo^{\lambda_{i_{1}+1}\circ\cdots\circ
\lambda_{i_{2}}} & G_{i_{2}} & \lTo^{\lambda_{i_{2}+1}\circ
\cdots\circ\lambda_{i_{3}}} & \cdots.
\end{diagram}
\]
In the future we denote a composition $\lambda_{i}\circ\cdots\circ\lambda_{j}$
($i\leq j$) by $\lambda_{i,j}$.

Sequences $\left\{  G_{i},\lambda_{i}\right\}  $ and $\left\{  H_{i},\mu
_{i}\right\}  $ are \emph{pro-isomorphic} if, after passing to subsequences,
there exists a commuting \textquotedblleft ladder diagram\textquotedblright:
\begin{equation}
\begin{diagram} G_{i_{0}} & & \lTo^{\lambda_{i_{0}+1,i_{1}}} & & G_{i_{1}} & & \lTo^{\lambda_{i_{1}+1,i_{2}}} & & G_{i_{2}} & & \lTo^{\lambda_{i_{2}+1,i_{3}}}& & G_{i_{3}}& \cdots\\ & \luTo & & \ldTo & & \luTo & & \ldTo & & \luTo & & \ldTo &\\ & & H_{j_{0}} & & \lTo^{\mu_{j_{0}+1,j_{1}}} & & H_{j_{1}} & & \lTo^{\mu_{j_{1}+1,j_{2}}}& & H_{j_{2}} & & \lTo^{\mu_{j_{2}+1,j_{3}}} & & \cdots \end{diagram} \label{basic ladder diagram}%
\end{equation}
Clearly an inverse sequence is pro-isomorphic to any of its subsequences. To
avoid tedious notation, we sometimes do not distinguish $\left\{
G_{i},\lambda_{i}\right\}  $ from its subsequences. Instead we assume that
$\left\{  G_{i},\lambda_{i}\right\}  $ has the desired properties of a
preferred subsequence---prefaced by the words \textquotedblleft after passing
to a subsequence and relabeling\textquotedblright.

An inverse sequence $\left\{  G_{i},\lambda_{i}\right\}  $ is called
\emph{pro-monomorphic }if it is pro-isomorphic to an inverse sequence of
monomorphisms and \emph{pro-epimorphic} (more commonly called \emph{semistable
}or \emph{Mittag-Leffler}) if it is pro-isomorphic to an inverse sequence of
epimorphisms. It is \emph{stable} if it is pro-isomorphic to a constant
inverse sequence $\left\{  H,\operatorname{id}_{H}\right\}  $, or
equivalently, to an inverse sequence of isomorphisms. It is a standard fact
that $\left\{  G_{i},\lambda_{i}\right\}  $ is stable if and only if it is
both pro-monomorphic and pro-epimorphic.

A few more special classes of inverse sequences will be of interest in this
paper. A sequence that is pro-isomorphic to the trivial sequence
$1\leftarrow1\leftarrow1\leftarrow\cdots$ is called \emph{pro-trivial}; a
sequence pro-isomorphic to an inverse sequence of finitely generated groups is
called \emph{pro-finitely generated}; and a sequence that is pro-isomorphic to
an inverse sequence of free groups is called \emph{pro-free}. A sequence that
is both pro-finitely generated and pro-free is easily seen to be
pro-isomorphic to an inverse sequence of finitely generated free groups. We
call such a sequence \emph{pro-finitely generated free}.

The \emph{inverse limit }of a sequence $\left\{  G_{i},\lambda_{i}\right\}  $
is the subgroup of $\prod G_{i}$ defined by
\[
\underleftarrow{\lim}\left\{  G_{i},\lambda_{i}\right\}  =\left\{  \left.
\left(  g_{0},g_{1},g_{2},\cdots\right)  \in\prod_{i=0}^{\infty}%
G_{i}\right\vert \lambda_{i}\left(  g_{i}\right)  =g_{i-1}\right\}  .
\]

\noindent In the special case where $\left\{  G_{i},\lambda_{i}\right\}  $ is
an inverse sequence of abelian groups, we also define the \emph{derived
limit}\footnote{The definition of derived limit can be generalized to include
nonableian groups (see \cite[\S 11.3]{Ge2}), but that will not be needed in
this paper.} to be the following quotient group:%
\[
\underleftarrow{\lim}^{1}\left\{  G_{i},\lambda_{i}\right\}  =\left(
\prod\limits_{i=0}^{\infty}G_{i}\right)  /\left\{  \left.  \left(
g_{0}-\lambda_{1}g_{1},g_{1}-\lambda_{2}g_{2},g_{2}-\lambda_{3}g_{3}%
,\cdots\right)  \right\vert \ g_{i}\in G_{i}\right\}
\]

It is a standard fact that pro-isomorphic inverse sequences of groups have
isomorphic inverse limits and, pro-isomorphic inverse sequences of abelian
groups have isomorphic derived limits.

\subsection{Absolute neighborhood retracts}

Throughout this paper, all spaces are assumed to be separable metric. A
locally compact space $X$ is an ANR (\emph{absolute neighborhood retract}) if
it can be embedded into $%
\mathbb{R}
^{n}$ or, if necessary, $%
\mathbb{R}
^{\infty}$ (a countable product of real lines) as a closed set in such a way
that there exists a retraction $r:U\rightarrow X$, where $U$ is a neighborhood
of $X$. If the entire space $%
\mathbb{R}
^{n}$ or $%
\mathbb{R}
^{\infty}$ retracts onto $X$, we call $X$ an AR (absolute retract). If $X$ is
finite-dimensional, all mention of $%
\mathbb{R}
^{\infty}$ can be omitted. A finite-dimensional ANR is called an ENR
(\emph{Euclidean neighborhood retract}) and a finite-dimensional AR an ER.
With a little effort it can be shown that an AR [resp., ER] is simply a
contractible ANR [resp., ENR].

A space $X$ is \emph{locally contractible }if every neighborhood $U$ of a
point $x\in X$ contains a neighborhood $V$ of $x$ that contracts within $U$.
It is a standard fact that every ANR is locally contractible. For
finite-dimensional spaces, that property characterizes ANRs. In other words, a
locally compact, finite-dimensional space $X$ is an ANR (and hence an ENR) if
and only if it is locally contractible. It follows that every
finite-dimensional locally finite polyhedron or CW complex is an ENR; if it is
contractible, it is an ER.

The following famous result will be used in this paper.

\begin{theorem}
[West, \cite{We}]\label{Theorem: West's Theorem}Every ANR is homotopy
equivalent to a locally finite polyhedron. Every compact ANR is homotopy
equivalent to a finite polyhedron.
\end{theorem}

\subsection{Proper maps and homotopies}

When working with noncompact space, the notion of `properness' is crucial. A
map $f:X\rightarrow Y$ is \emph{proper} if $f^{-1}\left(  C\right)  $ is
compact whenever $C\subseteq Y$ is compact. Maps $f_{0},f_{1}:X\rightarrow Y$
are \emph{properly homotopic}, denoted $f_{0}\overset{p}{\simeq}f_{1}$ if
there exists a proper map $H:X\times\lbrack0,1]\rightarrow Y$ with
$H_{0}=f_{0}$ and $H_{1}=f_{1}$. Spaces $X$ and $Y$ are \emph{proper homotopy
equivalent}, denoted $X\overset{p}{\simeq}Y$, if there exist proper maps
$f:X\rightarrow Y$ and $g:Y\rightarrow X$ with $gf\overset{p}{\simeq
}\operatorname*{id}_{X}$ and $fg\overset{p}{\simeq}\operatorname*{id}_{Y}$.

\subsection{Ends of spaces and the fundamental group at infinity}

A subset $N$ of a space $X$ is a \emph{neighborhood of infinity} if
$\overline{X-N}$ is compact. A standard argument shows that, when $X$ is an
ANR and $C$ is a compact subset of $X$, $X-C$ has at most finitely many
unbounded components, i.e., finitely many components with noncompact closures.
If $X-C$ has both bounded and unbounded components, the situation can be
simplified by letting $C^{\prime}$ consist of $C$ together with all bounded
components. Then $C^{\prime}$ is compact, and $X-C^{\prime}$ consists entirely
of unbounded components.

We say that $X$ \emph{has }$k$ \emph{ends }if there exists a compactum
$C\subseteq X$ such that, for every compactum $D$ with $C\subset D$, $X-D$ has
exactly $k$ unbounded components. When $k$ exists, it is uniquely determined;
if $k$ does not exist, we say $X$ has \emph{infinitely many ends}. Thus, a
space is $0$-ended if and only if $X$ is compact, and $1$-ended if and only if
it contains arbitrarily small connected neighborhoods of infinity.

A nested sequence $N_{0}\supseteq N_{1}\supseteq N_{2}\supseteq\cdots$ of
neighborhoods of infinity, with each $N_{i}\subseteq\operatorname*{int}%
N_{i-1}$, is \emph{cofinal }if $\bigcap_{i=0}^{\infty}N_{i}=\varnothing$. Such
a sequence is easily obtained: choose an exhaustion of $X$ by compacta
$C_{0}\subseteq C_{1}\subseteq C_{2}\subseteq\cdots$, with $C_{i-1}%
\subseteq\operatorname*{int}C_{i}$; then let $N_{i}=X-C_{i}$. When closed
neighborhoods of infinity are required, let $N_{i}=\overline{X-C_{i}}$.

Given a nested cofinal sequence $\left\{  N_{i}\right\}  _{i=0}^{\infty}$ of
neighborhoods of infinity, base points $p_{i}\in N_{i}$, and paths
$r_{i}\subset N_{i}$ connecting $p_{i}$ to $p_{i+1}$, we obtain an inverse
sequence:
\begin{equation}
\pi_{1}\left(  N_{0},p_{0}\right)  \overset{\lambda_{1}}{\longleftarrow}%
\pi_{1}\left(  N_{1},p_{1}\right)  \overset{\lambda_{2}}{\longleftarrow}%
\pi_{1}\left(  N_{2},p_{2}\right)  \overset{\lambda_{3}}{\longleftarrow}%
\cdots.\medskip\label{Defn: pro-pi1}%
\end{equation}
Here, each $\lambda_{i+1}:\pi_{1}\left(  N_{i+1},p_{i+1}\right)
\rightarrow\pi_{1}\left(  N_{i},p_{i}\right)  $ is the homomorphism induced by
inclusion followed by the change of base point isomorphism determined by
$r_{i}$. The proper ray $r:[0,\infty)\rightarrow X$ obtained by piecing
together the $r_{i}$ in the obvious manner is referred to as the \emph{base
ray }for the inverse sequence, and the pro-isomorphism class of the inverse
sequence is called the \emph{fundamental group at infinity of }$X$ \emph{based
at} $r$ and is denoted $\operatorname{pro}$-$\pi_{1}\left(  \varepsilon
(X),r\right)  $. It is a standard fact that $\operatorname{pro}$-$\pi
_{1}\left(  X,r\right)  $ is independent of the sequence of neighborhoods
$\left\{  N_{i}\right\}  $ or the base points---provided those base points
tend to infinity along the ray $r$, and corresponding subpaths of $r$ are used
in defining the $\lambda_{i}$. More generally, $\operatorname{pro}$-$\pi
_{1}\left(  \varepsilon(X),r\right)  $ depends only upon the proper homotopy
class of $r$. If $X$ is 1-ended and $\operatorname{pro}$-$\pi_{1}\left(
\varepsilon(X),r\right)  $ is semistable for some proper ray $r$, it can be
shown that all proper rays in $X$ are properly homotopic; in that case we say
that $X$ is \emph{strongly connected at infinity}. When $X$ is strongly
connected at infinity, it is safe to omit mention of the base ray and to speak
generally of the \emph{fundamental group at infinity of }$X$, and denote it by
$\operatorname{pro}$-$\pi_{1}\left(  \varepsilon(X)\right)  $. If $X$ is
1-ended and $\operatorname{pro}$-$\pi_{1}\left(  \varepsilon(X),r\right)  $ is
pro-trivial, we call $X$ \emph{simply connected at infinity}.

The fundamental group at infinity is clearly not a homotopy invariant of a
space, but it is a proper homotopy invariant. More precisely, if
$f:X\rightarrow Y$ is a proper homotopy equivalence, then $\operatorname{pro}%
$-$\pi_{1}\left(  \varepsilon(X),r\right)  $ is pro-isomorphic to
$\operatorname{pro}$-$\pi_{1}\left(  \varepsilon(Y),f\circ r\right)  $.

For a group $G$ of type F, the universal cover $\widetilde{K}$ of a finite
K$\left(  G,1\right)  $ complex $K$ is well-defined up to proper homotopy
type. So the number of ends of $G$ is well-defined; and if $\widetilde{K}$ is
1-ended, except for the issue of a base ray, we may view $\operatorname{pro}%
$-$\pi_{1}\left(  \varepsilon(\widetilde{K}),r\right)  $ as an invariant of
$G$. The base ray issue goes away when $\operatorname{pro}$-$\pi_{1}\left(
\varepsilon(\widetilde{K}),r\right)  $ is semistable, so there is no ambiguity
in defining a 1-ended $G$ to have semistable, stable, or trivial fundamental
group at infinity, according to whether $\operatorname{pro}$-$\pi_{1}\left(
\varepsilon(\widetilde{K}),r\right)  $ has the corresponding property. With
some additional work, it can be shown that the property of $\operatorname{pro}%
$-$\pi_{1}\left(  \varepsilon(\widetilde{K}),r\right)  $ being pro-monomorphic
is also independent of base ray and, thus, attributable to $G$. See
\cite[\S 2]{GG} for further discussion.

Although not needed for this paper, the requirement in the previous paragraph,
that $G$ have type F can be significantly weakened. In particular, if $G$ is
finitely presented, and $L$ is any finite complex with fundamental group $G$,
then the number of ends of $\widetilde{L}$ and the properties of
$\operatorname{pro}$-$\pi_{1}\left(  \varepsilon(\widetilde{L}),r\right)  $
discussed above, are invariants of $G$. Thus, for example, a finitely
presented group $G$ is called \emph{simply connected at infinity} if
$\widetilde{L}$ has that property. For more information about the fundamental
group at infinity of spaces and groups, including proofs of the made in this
section, see \cite{Ge2} or \cite{Gu3}.

\subsection{Finite domination and inward tameness}

A space $Y$ has \emph{finite homotopy type }if it is homotopy equivalent to a
finite CW complex; it is \emph{finitely dominated }if there is a finite
complex $K$ and maps $u:Y\rightarrow K$ and $d:K\rightarrow Y$ such that
$d\circ u\simeq\operatorname*{id}_{Y}$. If $Y$ is an ANR, then $Y$ is finitely
dominated if and only if there exists a self-homotopy that `pulls $Y$ into a
compact subset', i.e., $H:Y\times\lbrack0,1]\rightarrow Y$ such that
$H_{0}=\operatorname*{id}_{Y}$ and $\overline{H_{1}\left(  Y\right)  }$ is
compact. This equivalence is easily verified when (for example) $K$ is a
locally finite polyhedron; a discussion of the general case can be found in
\cite[\S 3.4]{Gu3}.

The following clever observation will be used later.

\begin{theorem}
[Mather, \cite{Ma}]\label{Theorem: Mather's theorem}If a space $Y$ is finitely
dominated, then $Y\times\mathbb{S}^{1}$ has finite homotopy type.
\end{theorem}

An ANR $X$ is \emph{inward tame} if, for every closed neighborhood of infinity
$N$ in $X$, there is a homotopy $K:N\times\left[  0,1\right]  \rightarrow N$
with $K_{0}=\operatorname*{id}_{N}$ and $\overline{K_{1}\left(  N\right)  }$
compact (a homotopy pulling $N$ into a compact subset). By an easy application
of Borsuk's Homotopy Extension Property, this is equivalent to the existence
of a cofinal sequence $\{N_{i}\}$ of closed neighborhoods of infinity, each of
which can be pulled into a compact set. If $X$ contains a cofinal sequence
$\left\{  N_{i}\right\}  $ of closed ANR neighborhoods\footnote{In this case,
$X$ is called \emph{sharp at infinity}. Most commonly arising ANRs, for
example: locally finite polyhedra, manifolds, proper CAT(0) spaces, and
Hilbert cube manifolds) are sharp at infinity.} of infinity, then inward
tameness is equivalent to each of those (hence, all closed ANR neighborhoods
of infinity) being finitely dominated. \cite[\S 3.5]{Gu3} provides additional details.

Inward tameness is an invariant of proper homotopy type. Roughly speaking, if
$f:X\rightarrow Y$ and $g:Y\rightarrow X$ are proper homotopy inverses and $H$
is a homotopy that pulls a neighborhood of infinity of $X$ into a compact set,
then $f\circ H_{t}\circ g$ pulls a neighborhood of $Y$ into a compact set.
More details can be found in \cite[\S 3.5]{Gu3}.

\subsection{Some basic K-theory}

An important result from \cite{Wa} asserts that, for each finitely dominated,
connected space $Y$, there is a well-defined obstruction $\sigma\left(
Y\right)  $, lying in the \emph{reduced projective class group} $\widetilde
{K}_{0}\!\left(
\mathbb{Z}
\left[  \pi_{1}\left(  Y\right)  \right]  \right)  $, which vanishes if and
only if $Y$ has finite homotopy type.

A related algebraic construction is the \emph{Whitehead group}. If $\left(
A,B\right)  $ is a pair of connected, finite CW complexes and
$B\hookrightarrow A$ is a homotopy equivalence, then there is a well-defined
obstruction $\tau\left(  B\right)  $, lying in an abelian group
$\operatorname*{Wh}\!\left(  \pi_{1}\left(  B\right)  \right)  $ that vanishes
if and only if $B\hookrightarrow A$ is a simple homotopy equivalence.
Definitions and details can be found in \cite{Co}.

Both of the above algebraic constructs act as functors in the sense that, if
$\lambda:G\rightarrow H$ is a group homomorphism, there are naturally induced
homomorphims $\lambda_{\ast}:\widetilde{K}_{0}\!\left(
\mathbb{Z}
\left[  G\right]  \right)  \rightarrow\widetilde{K}_{0}\!\left(
\mathbb{Z}
\left[  G\right]  \right)  $ and $\lambda_{\ast}:\operatorname*{Wh}\!\left(
G\right)  \rightarrow\operatorname*{Wh}\!\left(  H\right)  $.

For the purposes of this paper, the main thing we need to know about
$\widetilde{K}_{0}\!\left(
\mathbb{Z}
\left[  \pi_{1}\left(  Y\right)  \right]  \right)  $ or $\operatorname*{Wh}%
\!\left(  \pi_{1}\left(  B\right)  \right)  $ is contained in a famous result
of Bass, Heller and Swan \cite{BHS}.

\begin{theorem}
\label{Theorem: Bass-Heller-Swan}If $G$ is a finitely generated free group,
then both $\widetilde{K}_{0}\!\left(
\mathbb{Z}
\left[  G\right]  \right)  $ and $\operatorname*{Wh}\!\left(  G\right)  $ are
the trivial group.
\end{theorem}

\subsection{Mapping cylinders, mapping tori, and mapping telescopes}

For any map $f:K\rightarrow L$ and closed interval $\left[  a,b\right]  $, the
\emph{mapping cylinder} $\mathcal{M}_{\left[  a,b\right]  }\left(  f\right)  $
is the quotient space $L\sqcup\left(  K\times\left[  a,b\right]  \right)
/\!\sim$, where $\sim$ is the equivalence relation generated by the rule
$\left(  x,a\right)  \sim f\left(  x\right)  $ for all $x\in K$. Let
$q_{\left[  a,b\right]  }:L\sqcup(K\times\left[  a,b\right]  )\rightarrow
\mathcal{M}_{\left[  a,b\right]  }\left(  f\right)  $ be the quotient map.
Then, for each $r\in(a,b]$, $q_{\left[  a,b\right]  }$ restricts to an
embedding of $K\times\left\{  r\right\}  $ into $\mathcal{M}_{\left[
a,b\right]  }\left(  f\right)  $; denote the image of $K\times\left\{
r\right\}  $ by $K_{r}$. The quotient map is also an embedding when restricted
to $L$; let $L_{a}\subseteq\mathcal{M}_{\left[  a,b\right]  }\left(  f\right)
$ be that copy of $L$. We call $K_{b}$ the \emph{domain end} and $L_{a}$ the
\emph{range end }of $\mathcal{M}_{\left[  a,b\right]  }\left(  f\right)  $.
Note the existence of a projection map $p_{\left[  a,b\right]  }%
:\mathcal{M}_{\left[  a,b\right]  }\left(  f\right)  \rightarrow\left[
a,b\right]  $ for which $p_{\left[  a,b\right]  }^{-1}\left(  r\right)
=K_{r}$ is a copy of $K$ for each $r\in(a,b]$ and $p_{\left[  a,b\right]
}^{-1}\left(  a\right)  =L_{a}$ is a copy of $L$. Note also that, when $K=L$,
i.e., $f$ maps $K$ to itself, all of the above still applies. In that case,
each point preimage of $p_{\left[  a,b\right]  }$ is a copy of $K$, but the
copy $K_{a}$ differs from the others, in that it is not necessarily parallel
to neighboring copies.

\begin{remark}
\emph{Clearly the topological type of }$\mathcal{M}_{\left[  a,b\right]
}\left(  f\right)  $\emph{ does not depend on the interval }$\left[
a,b\right]  $\emph{, and for most purposes can be taken to be }$\left[
0,1\right]  $\emph{. But in the treatment that follows, it will be useful to
allow the interval to vary.}
\end{remark}

The following standard application of mapping cylinders will be used several
times in this paper. A proof, in which properness is not mentioned, can be
found in \cite[p.372]{Du}. For our purposes, it is only the easy (converse)
direction of the proper assertion that will be used.

\begin{lemma}
\label{Lemma: mapping cylinders of homotopy equivalences}A map $f:K\rightarrow
L$ between ANRs is a homotopy equivalence if and only if there exists a strong
deformation retraction of $\mathcal{M}_{\left[  a,b\right]  }\left(  f\right)
$ onto $M_{b}$. It is a proper homotopy equivalence if and only if there
exists a proper strong deformation retraction of $\mathcal{M}_{\left[
a,b\right]  }\left(  f\right)  $ onto $K_{b}$.
\end{lemma}

The \emph{bi-infinite mapping telescope} of a map $f:K\rightarrow K$ is
obtained by gluing together infinitely many mapping cylinders. More
precisely,
\[
\operatorname*{Tel}\nolimits_{f}\left(  K\right)  =\cdots\cup\mathcal{M}%
_{\left[  -2,-1\right]  }\left(  f\right)  \cup\mathcal{M}_{\left[
-1,0\right]  }\left(  f\right)  \cup\mathcal{M}_{\left[  0,1\right]  }\left(
f\right)  \cup\mathcal{M}_{\left[  1,2\right]  }\left(  f\right)
\cup\mathcal{M}_{\left[  2,3\right]  }\left(  f\right)  \cup\cdots
\]
where the gluing is accomplished by identifying the domain end of each
$\mathcal{M}_{\left[  n-1,n\right]  }\left(  f\right)  $ with the range end of
$\mathcal{M}_{\left[  n,n+1\right]  }\left(  f\right)  $. Notationally, this
works well since, under the convention described above, each is denoted
$K_{n}$. Projection maps may be pieced together to obtain a projection
$p:\operatorname*{Tel}\nolimits_{f}\left(  K\right)  \rightarrow%
\mathbb{R}
$, for which $p^{-1}\left(  r\right)  =K_{r}$ is a copy of $K$, for each $r\in%
\mathbb{R}
$. A schematic of $\operatorname*{Tel}\nolimits_{f}\left(  K\right)  $ is
contained in Figure \ref{Fig 1} of
\S \ref{Subsection: Mapping tori of self-homotopy equivalences}.

The \emph{mapping torus} of $f:K\rightarrow K$ is obtained from $\mathcal{M}%
_{\left[  0,1\right]  }\left(  f\right)  $ by identifying $K_{0}$ with $K_{1}%
$. It may also be defined more directly as the quotient space%
\[
\operatorname*{Tor}\nolimits_{f}\left(  K\right)  =K\times\left[  0,1\right]
/\sim
\]
where $\sim$ is the equivalence relation generated by $\left(  x,0\right)
\sim\left(  f\left(  x\right)  ,1\right)  $ for each $x\in K$. The following
facts about mapping mapping tori are standard.

\begin{lemma}
Let $K$ be a connected ANR, $f:\left(  K,p\right)  \rightarrow\left(
K,q\right)  $ a map that induces an isomorphism $\varphi:\pi_{1}\left(
K,p\right)  \rightarrow\pi_{1}\left(  K,q\right)  $, and $\lambda$ a path in
$K$ from $q$ to $p$. Then

\begin{enumerate}
\item $\pi_{1}\left(  \operatorname*{Tor}\nolimits_{f}\left(  K\right)
,(p,0\right)  )\allowbreak\cong\allowbreak\pi_{1}\left(  K,p\right)
\rtimes_{\varphi}\left\langle t\right\rangle $, where $t$ is an infinite order
element represented by the loop $\left(  \left\{  p\right\}  \times\left[
0,1\right]  \right)  \cdot\lambda$, and

\item the infinite cyclic cover of $\operatorname*{Tor}\nolimits_{f}\left(
K\right)  $ corresponding to the projection $\pi_{1}\left(  K,p\right)
\rtimes_{\varphi}\left\langle t\right\rangle \rightarrow\left\langle
t\right\rangle $ is the bi-infinite mapping telescope $\operatorname*{Tel}%
\nolimits_{f}\left(  K\right)  $.
\end{enumerate}
\end{lemma}

The following fact about mapping tori can be found in \cite{GG}, where it
plays a crucial role. We will make significant use of it here as well.

\begin{lemma}
\label{Lemma: mapping torus/Z-action}Let $X$ be a connected ANR that admits a
proper $%
\mathbb{Z}
$-action generated by a homeomorphism $j:X\rightarrow X$. Then $(\left\langle
j\right\rangle \backslash X)\times\mathbb{R}$ is homeomorphic to
$\operatorname*{Tor}\nolimits_{j}\left(  X\right)  $.
\end{lemma}

\subsection{Hilbert cube manifolds}

The \emph{Hilbert cube} is the infinite product
\[
\mathcal{Q=}\prod_{i=1}^{\infty}\left[  -1,1\right]  \text{, with metric
}d\left(  \left(  x_{i}\right)  ,\left(  y_{i}\right)  \right)  =\sum
_{i=1}^{\infty}\frac{\left\vert x_{i}-y_{i}\right\vert }{2^{i}}%
\]
A \emph{Hilbert cube manifold} is a space $X$ with the property that each
$x\in X$ has a neighborhood homeomorphic to $\mathcal{Q}$. Although we are
primarily interested in finite-dimensional spaces, Hilbert cube manifolds play
a key role in this paper. A pair of classical results will allow us to move
between the categories of ANRs and locally finite polyhedra through the use of
Hilbert cube manifolds.

\begin{theorem}
[Edwards, \cite{Ed}]\label{Theorem: Edwards HCM Theorem}If $A$ is an ANR, then
$A\times\mathcal{Q}$ is a Hilbert cube manifold.
\end{theorem}

\begin{theorem}
[Chapman, \cite{Ch}]\label{Theorem: Chapmans triangulation of HCMs}If $X$ is a
Hilbert cube manifold, then there is a locally finite polyhedron $K$ for which
$X\approx K\times\mathcal{Q}$.
\end{theorem}

\subsection{$\mathcal{Z}$-sets and $\mathcal{Z}$%
-com\-pact\-ific\-at\-ions\label{Subsection: Z-sets}}

A closed subset $A$ of an ANR\ $Y$ is a $\mathcal{Z}$\emph{-set} if either of
the following equivalent conditions is satisfied:

\begin{itemize}
\item There exists a homotopy $H:Y\times\left[  0,1\right]  \rightarrow Y$
such that $H_{0}=\operatorname*{id}_{Y}$ and $H_{t}\left(  X\right)  \subseteq
Y-A$ for all $t>0$. (We say that $H$ \emph{instantly homotopes} $Y$ off from
$A$.)

\item For every open set $U$ in $Y$, $U-A\hookrightarrow U$ is a homotopy equivalence.
\end{itemize}

\noindent A $\mathcal{Z}$\emph{-com\-pact\-ific\-at\-ion} of a space $X$ is a
com\-pact\-ific\-at\-ion $\overline{X}=X\sqcup Z$ with the property that $Z$
is a $\mathcal{Z}$-set in $\overline{X}$. In this case, $Z$ is called a
$\mathcal{Z}$\emph{-boundary }for $X$. Implicit in this definition is the
requirement that $\overline{X}$ be an ANR; moreover, since an open subset of
an ANR is an ANR, $X$ itself must be an ANR to be a candidate for
$\mathcal{Z}$-com\-pact\-ific\-at\-ion. Hanner's Theorem \cite{Ha} ensures
that every com\-pact\-ific\-at\-ion $\overline{X}$ of an ANR $X$, for which
$\overline{X}-X$ satisfies either of the \textquotedblleft negligibility
conditions\textquotedblright\ in the definition of $\mathcal{Z}$-set, is
necessarily an ANR; hence, it is a $\mathcal{Z}$-com\-pact\-ific\-at\-ion. By
a similar (but much easier) result in dimension theory, $\mathcal{Z}%
$-com\-pact\-ific\-at\-ion does not increase dimension; so, if $X$ is an ENR,
so is $\overline{X}$.

\begin{example}
The com\-pact\-ific\-at\-ion of $%
\mathbb{R}
^{n}$ obtained by adding the $\left(  n-1\right)  $-sphere at infinity is the
prototypical $\mathcal{Z}$-com\-pact\-ific\-at\-ion. More generally, addition
of the visual boundary to a proper CAT(0) space is a $\mathcal{Z}%
$-com\-pact\-ific\-at\-ion. In \cite{BM}, it is shown that, for a torsion-free
$\delta$-hyperbolic group $G$, an appropriately chosen Rips complex can be
$\mathcal{Z}$-compactified by adding the Gromov boundary $\partial G$.
\end{example}

The purely topological question of when an ANR, an ENR, or even a locally
finite polyhedron admits a $\mathcal{Z}$-com\-pact\-ific\-at\-ion is an open
question (see \cite{Gu1}). However, Chapman and Siebenmann \cite{CS} have
given a complete classification of $\mathcal{Z}$-compactifiability for Hilbert
cube manifolds. That result, in combination with Theorem
\ref{Theorem: Edwards HCM Theorem}, has substantial implications for the
general case.

Here we provide a slightly simplified version of the Chapman-Siebenmann
theorem. We state the result only for 1-ended Hilbert cube manifolds $X$,
since that is all we need in this paper. We also simplify the definitions of
$\sigma_{\infty}\left(  X\right)  $ and $\tau_{\infty}\left(  X\right)  $ by
basing them on a prechosen nested, cofinal sequence of nice neighborhoods of
infinity. It is true, but would take some time, to explain why the resulting
obstructions do not depend on that choice.

A particularly nice variety of closed neighborhood of infinity $N\subseteq X$
is one that is, itself, a Hilbert cube manifold and whose topological boundary
$\operatorname*{Bd}_{X}N$ is a Hilbert cube manifold with a neighborhood in
$X$ homeomorphic to $\operatorname*{Bd}_{X}N\times\lbrack-1,1]$. Call such
neighborhoods of infinity \emph{clean}. By applying Theorems
\ref{Theorem: Edwards HCM Theorem} and
\ref{Theorem: Chapmans triangulation of HCMs}, clean neighborhoods of infinity
are easily found.

\begin{theorem}
[Chapman-Siebenmann]\label{CS Theorem}Let $X$ be a 1-ended Hilbert cube
manifold and $\left\{  N_{i}\right\}  $ a nested cofinal sequence of connected
clean neighborhoods of infinity. Then $X$ admits a $\mathcal{Z}$%
-com\-pact\-ific\-at\-ion if and only if each of the following conditions holds:

\begin{enumerate}
\item[a)] $X$ is inward tame.

\item[b)] $\sigma_{\infty}(X)\in\underleftarrow{\lim}\left\{  \widetilde
{K}_{0}(%
\mathbb{Z}
\pi_{1}(N_{i})),\lambda_{i\ast}\right\}  $ is zero.

\item[c)] $\tau_{\infty}\left(  X\right)  \in\underleftarrow{\lim}^{1}\left\{
\operatorname*{Wh}(\pi_{1}(N_{i})),\lambda_{i\ast}\right\}  $ is zero.
\end{enumerate}
\end{theorem}

The inverse sequences $\left\{  \widetilde{K}_{0}(%
\mathbb{Z}
\pi_{1}(N_{i})),\lambda_{i\ast}\right\}  $ and $\left\{  \operatorname*{Wh}%
(\pi_{1}(N_{i})),\lambda_{i\ast}\right\}  $ in conditions b) and c) are
obtained by applying the $\widetilde{K}_{0}$-functor and the
$\operatorname*{Wh}$-functor to sequence (\ref{Defn: pro-pi1}). The
obstruction $\sigma_{\infty}(X)$ is just the sequence $\left(  \sigma\left(
N_{0}\right)  ,\sigma\left(  N_{1}\right)  ,\sigma\left(  N_{2}\right)
,\cdots\right)  $ of Wall finiteness obstructions of the $N_{i}$. Condition a)
ensures that each $N_{i}$ is finitely dominated, so the individual
obstructions are all defined; without condition a), there is no condition b).
Similarly, condition c) requires condition b). It is related to the Whitehead
torsion of inclusions $\operatorname*{Bd}_{X}N_{i}\hookrightarrow
\overline{N_{i}-N_{i+1}}$, after the $N_{i}$ have been improved significantly
so that those inclusions are homotopy equivalences. The reader should consult
\cite{CS} for details or \cite[\S 8.2]{Gu3} for a less formal discussion of
Theorem \ref{CS Theorem}. For our purposes, those details are not so important
since the obstructions arising here will be shown to vanish by straightforward
applications of Theorem \ref{Theorem: Bass-Heller-Swan}.

\begin{remark}
\emph{Condition a) makes sense for an arbitrary ANR }$X$\emph{. If }$X$\emph{
satisfies a) and is sharp at infinity, then condition b) also makes immediate
sense; it is satisfied if and only if }$X$\emph{ contains arbitrarily small
closed ANR neighborhoods of infinity having finite homotopy type. Condition c)
is more problematic; even when a) and b) are satisfied, if }$X$\emph{ is not a
Hilbert cube manifold, it may be impossible to find neighborhoods of infinity
}$N_{i}$\emph{ with each }$\operatorname*{Bd}_{X}N_{i}\hookrightarrow
\overline{N_{i}-N_{i+1}}$\emph{ a homotopy equivalence---an example can be
found in \cite{GT}. The solution to this problem is to} define \emph{the
obstructions for an ANR }$X$\emph{ to be the corresponding obstructions for
the Hilbert cube manifold} $X\times\mathcal{Q}$. \emph{Then a)-c) are
necessary for }$\mathcal{Z}$\emph{-compactifiability of }$X$\emph{;
unfortunately, they are not sufficient. \cite{Gu1} exhibits a locally finite
2-dimensional polyhedron }$K$\emph{ that satisfies a)-c), but is not
}$\mathcal{Z}$\emph{-com\-pact\-ifi\-able. A suitable characterization of
}$\mathcal{Z}$\emph{-com\-pact\-ifi\-able ANRs is an open question.}
\end{remark}

For an ANR $X$, Theorem \ref{CS Theorem} allows us to determine whether
$X\times\mathcal{Q}$ is $\mathcal{Z}$-com\-pact\-ifi\-able. The following
result, with $\mathbb{I=}\left[  -1,1\right]  $, frequently allows us to
restore finite-dimensionality.

\begin{theorem}
[Ferry, \cite{Fe}]\label{Theorem: Ferry's stabilization theorem}If $K$ is a
finite-dimensional locally finite polyhedron and $K\times\mathcal{Q}$ is
$\mathcal{Z}$-com\-pact\-ifi\-able, then $K\times\mathbb{I}^{2\cdot\dim K+5}$
is $\mathcal{Z}$-com\-pact\-ifi\-able.
\end{theorem}

\section{Topological results\label{Section: Topological results}}

In this section we prove a variety of topological results that are primary
ingredients in the proofs of our main theorems. We have broken the section
into three parts: the first contains results about product spaces; the second
deals with spaces that admit a proper $%
\mathbb{Z}
$-action generated by a homeomorphism properly homotopic to the identity; and
the third looks at spaces that are simply connected at infinity.

\subsection{Products of noncompact spaces}

\begin{lemma}
\label{Lemma: f.d. cross R is inward tame}Let $X$ be an ANR that is finitely
dominated. Then $X\times%
\mathbb{R}
$ is inward tame.
\end{lemma}

\begin{proof}
Since inward tameness is an invariant of proper homotopy type, we may use
Theorem \ref{Theorem: West's Theorem} to reduce to the case that $X$ is a
locally finite polyhedron. For that case, the proof given in \cite[Prop.
3.1]{Gu2} for open manifolds is valid, with only minor modifications. With a
few additional modifications, the appeal to Theorem
\ref{Theorem: West's Theorem} can be eliminated.
\end{proof}

The next lemma requires a new definition. We say that an ANR $X$ is
\emph{movably finitely dominated} if, for every neighborhood of infinity
$N\subseteq X$, there is a self-homotopy of $X$ that pulls $X$ into a compact
subset of $N$, i.e., $H:X\times\lbrack0,1]\rightarrow X$ such that
$H_{0}=\operatorname*{id}_{X}$ and $\overline{H_{1}\left(  X\right)  }$ is
compact and contained in $N$. The motivation for this definition becomes
immediately clear in the following lemma. The most important examples are the
simplest---every contractible ANR is movably finitely dominated, since it is
dominated by each singleton subset.

\begin{lemma}
\label{Lemma: inward tameness of products}Let $X$ and $Y$ be connected,
noncompact, movably finitely dominated ANRs. Then $X\times Y$ is inward tame.
\end{lemma}

\begin{proof}
Let $A\subseteq X$ and $B\subseteq Y$ be compact and $N=\overline{(X\times
Y)-(A\times B)}$ the corresponding closed neighborhood of infinity. It
suffices to prove:\smallskip

\noindent\textbf{Claim. }\emph{There exits a homotopy }$J:N\times\left[
0,1\right]  \rightarrow N$\emph{ with }$J_{0}=\operatorname*{id}_{N}$\emph{
and }$\overline{J_{1}\left(  N\right)  }$\emph{ compact.}\smallskip

Choose compacta $A^{\prime}\subseteq X$ and $B^{\prime}\subseteq Y$ such that
$A\subseteq\operatorname*{int}_{X}A^{\prime}$ and $B\subseteq
\operatorname*{int}_{X}B^{\prime}$, and let $\lambda:X\rightarrow\left[
0,1\right]  $ and $\mu:Y\rightarrow\left[  0,1\right]  $ be Urysohn functions
with $\lambda\left(  A\right)  =0=\mu\left(  B\right)  $ and $\lambda\left(
\overline{X-A^{\prime}}\right)  =1=\mu\left(  \overline{Y-B^{\prime}}\right)
$. Then choose compact $K\subseteq X-A^{\prime}$ and $L\subseteq Y-B^{\prime}$
along with homotopies $F:X\times\left[  0,1\right]  \rightarrow X$ such that
$F_{0}=\operatorname*{id}_{X}$ and $F_{1}\left(  X\right)  \subseteq K$ and
$G:Y\times\left[  0,1\right]  \rightarrow Y$ such that $G_{0}%
=\operatorname*{id}_{Y}$ and $G_{1}\left(  X\right)  \subseteq L$.

We will build a homotopy $H$ that pulls $X\times Y$ into a compact subset
while fixing $A\times B$. By arranging that tracks of points from $N$ stay in
$N$, the restriction of this homotopy will satisfy the claim.

Define $\widehat{F}:X\times Y\times\left[  0,1\right]  \rightarrow X\times Y$
by $\widehat{F}\left(  x,y,t\right)  =\left(  F\left(  x,\mu\left(  y\right)
\cdot t\right)  ,y\right)  $ and note that:

\begin{itemize}
\item $\widehat{F}_{1}\left(  X\times Y\right)  \subseteq(X\times B^{\prime
})\cup\left(  K\times Y\right)  $,

\item $\left.  \widehat{F}_{t}\right\vert _{X\times B}=\operatorname*{id}$ for
all $t$, and

\item tracks of points in $N$ stay in $N$.
\end{itemize}

Similarly, let $\widehat{G}:X\times Y\times\left[  0,1\right]  \rightarrow
X\times Y$ by $\widehat{G}\left(  x,y,t\right)  =\left(  x,G\left(
y,\lambda\left(  x\right)  \cdot t\right)  \right)  $ and note that:

\begin{itemize}
\item $\widehat{G}_{1}\left(  X\times Y\right)  \subseteq(A^{\prime}\times
Y)\cup\left(  X\times L\right)  $,

\item $\left.  \widehat{G}\right\vert _{A\times Y}=\operatorname*{id}$, and

\item tracks of points in $N$ stay in $N$.
\end{itemize}

Now define $H:X\times Y\times\left[  0,1\right]  \rightarrow X\times Y$ by the
rule.%
\[
H_{t}=\left\{
\begin{array}
[c]{cc}%
\widehat{F}_{3t} & 0\leq t\leq\frac{1}{3}\\
\widehat{G}_{3t-1}\circ\widehat{F}_{1} & \frac{1}{3}\leq t\leq\frac{2}{3}\\
\widehat{F}_{3t-2}\circ\widehat{G}_{1}\circ\widehat{F}_{1} & \frac{2}{3}\leq
t\leq1
\end{array}
\right.
\]
A quick check shows that $H_{1}\left(  X\times Y\right)  $ is contained in the
compact set $\widehat{F}_{1}\widehat{G}_{1}(A^{\prime}\times B^{\prime}%
)\cup\left(  K\times L\right)  $; moreover, since the tracks of $H$ are all
concatenations of tracks of $\widehat{F}$ and $\widehat{G}$, $A\times B$ stays
fixed and tracks of points from $N$ stay in $N$. Letting $J$ be the
restriction of $H$ completes the claim.
\end{proof}

\begin{corollary}
\label{Corollary: products of noncompact ANRs are inward tame}The product of
any two\ noncompact ARs is inward tame.
\end{corollary}

\begin{lemma}
\label{Lemma: products that are fg free at infinity}Let $X$ and $Y$ be
noncompact, simply connected ANRs. Then $X\times Y$ contains arbitrarily small
path-connected neighborhoods of infinity, each with a fundamental group that
is finitely generated and free.
\end{lemma}

\begin{proof}
Let $U\subseteq X$ and $V\subseteq Y$ be open neighborhoods of infinity,
consisting of finitely many unbounded path-connected components $\left\{
U_{i}\right\}  _{i=1}^{k_{1}}$ and $\left\{  V_{j}\right\}  _{j=1}^{k_{2}}$,
respectively. Then the corresponding rectangular neighborhood of infinity
$R=(U\times Y)\cup\left(  X\times V\right)  $ may be covered by the finite
collection of path-connected open sets $\left\{  U_{i}\times Y\right\}
_{i=1}^{k_{1}}\cup\left\{  X\times V_{j}\right\}  _{j=1}^{k_{2}}$ in which
each of the two subcollections is pairwise disjoint, and each $U_{i}\times Y$
intersects each $X\times V_{j}$ in the path-connected set $U_{i}\times V_{j}$.
The nerve of this cover is the complete bipartite graph $K_{k_{1},k_{2}}$ and
the connectedness of this graph implies the path-connectedness of $R$. A
straight-forward application of the Generalized van Kampen Theorem to the
corresponding generalized graph of groups (see \cite[Th.6.2.11]{Ge2}) shows
that the fundamental group of $R$ is free on $\left(  k_{1}-1\right)  \left(
k_{2}-1\right)  $ generators, the key observation being that each element of a
vertex group $\pi_{1}\left(  U_{i}\times Y\right)  $ can be represented by a
loop in $U_{i}\times V_{j}$ which then contracts in $X\times V_{j}$, and
similarly for elements of vertex groups $\pi_{1}\left(  X\times V_{j}\right)
$.
\end{proof}

\begin{theorem}
\label{Theorem: products of HCMs}Let $X$ and $Y$ be noncompact, simply
connected, movably finitely dominated Hilbert cube manifolds. Then $X\times Y$
is $\mathcal{Z}$-com\-pact\-ifi\-able.
\end{theorem}

\begin{proof}
By a combination of Corollary
\ref{Corollary: products of noncompact ANRs are inward tame}, Lemma
\ref{Lemma: products that are fg free at infinity}, and Theorem
\ref{Theorem: Bass-Heller-Swan}, $X\times Y$ satisfies all conditions of
Theorem \ref{CS Theorem}.
\end{proof}

\begin{theorem}
Let $P_{1}$ and $P_{2}$ be noncompact, simply connected, moveably finitely
dominated, finite-dimensional, locally finite polyhedra. Then $P_{1}\times
P_{2}\times\mathbb{I}^{2(\dim P_{1}+\dim P_{2})+5}$ is $\mathcal{Z}$-com\-pact\-ifi\-able.
\end{theorem}

\begin{proof}
Apply Theorems \ref{Theorem: Edwards HCM Theorem} and
\ref{Theorem: products of HCMs} to $P_{1}\times P_{2}\times\mathcal{Q}$; then
use Theorem \ref{Theorem: Ferry's stabilization theorem}.
\end{proof}

\subsection{Spaces admitting homotopically simple $%
\mathbb{Z}
$-actions}

In this section we consider spaces $X$ that admit a proper $%
\mathbb{Z}
$-action generated by a homeomorphism properly homotopic to
$\operatorname*{id}_{X}$. Under the right circumstances, that hypothesis has
significant implication for the topology of $X$.

\begin{lemma}
\label{Lemma: inward tameness of spaces admitting nice Z-actions}Let $X$ be an
ANR that admits a proper $%
\mathbb{Z}
$-action generated by a homeomorphism $j:X\rightarrow X$ that is properly
homotopic to $\operatorname*{id}_{X}$. Then

\begin{enumerate}
\item if the action is cocompact, $X$ is 2-ended;

\item if the action is not cocompact, $X$ is 1-ended; and

\item if $X$ is finitely dominated, then $X$ is inward tame.
\end{enumerate}
\end{lemma}

\begin{proof}
By Lemma \ref{Lemma: mapping torus/Z-action}, $(\left\langle j\right\rangle
\backslash X)\times%
\mathbb{R}
\approx\operatorname*{Tor}_{j}\left(  X\right)  $, and since $j\overset
{p}{\simeq}\operatorname*{id}_{X}$, the latter space is proper homotopy
equivalent to $X\times\mathbb{S}^{1}$. Now $(\left\langle j\right\rangle
\backslash X)\times%
\mathbb{R}
$ has either two or one ends, according to whether $\left\langle
j\right\rangle \backslash X$ is compact or noncompact, and since the number of
ends is a proper homotopy invariant, the same is true for $X\times
\mathbb{S}^{1}$. Since $X\times\mathbb{S}^{1}$ has the same number of ends as
$X$, the first two assertions follow.

Next assume that $X$ is finitely dominated. By Theorem
\ref{Theorem: Mather's theorem}, $X\times\mathbb{S}^{1}$ has finite homotopy
type, so by the above equivalences, $\left\langle j\right\rangle \backslash X$
also has finite homotopy type. By Lemma
\ref{Lemma: f.d. cross R is inward tame}, $(\left\langle j\right\rangle
\backslash X)\times%
\mathbb{R}
$ is inward tame, and since inward tameness is an invariant of proper homotopy
type, $X\times\mathbb{S}^{1}$ is inward tame. It follows that $X$ is inward
tame since, if $N$ is a closed neighborhood of infinity in $X$, then
$N\times\mathbb{S}^{1}$ is a closed neighborhood of infinity in $X\times
\mathbb{S}^{1}$; and a homotopy that pulls $N\times\mathbb{S}^{1}$ into a
compact subset projects to a homotopy that pulls $N$ into a compact subset.
\end{proof}

\begin{lemma}
\label{Lemma: pro-pi1 of spaces admitting Z-actions}Let $X$ be a simply
connected ANR that admits a proper $%
\mathbb{Z}
$-action generated by a homeomorphism $j:X\rightarrow X$ that is properly
homotopic to $\operatorname{id}_{X}$. Then

\begin{enumerate}
\item if the action is cocompact, $X$ is simply connected at each of its two
ends, and

\item if the action is not cocompact, $X$ is strongly connected at infinity
and $\operatorname*{pro}$-$\pi_{1}\left(  \varepsilon\left(  X\right)
\right)  $ is pro-finitely generated free.
\end{enumerate}
\end{lemma}

\begin{proof}
The proof is primarily an application of work done in \cite{GM2}; we add a few
observations to make those results fit our situation more precisely. For both
assertions, we again use the equivalences:%
\begin{equation}
(\left\langle j\right\rangle \backslash X)\times%
\mathbb{R}
\approx\operatorname*{Tor}\nolimits_{j}\left(  X\right)  \overset{p}{\simeq
}X\times\mathbb{S}^{1}. \label{Equivalences: Quotient x R versus X x S1}%
\end{equation}

First assume that $\left\langle j\right\rangle \backslash X$ is compact. Then
$(\left\langle j\right\rangle \backslash X)\times%
\mathbb{R}
$ is 2-ended and the natural choices of base rays: $r_{-}=\left\{  p\right\}
\times(-\infty,0]$ and $r_{+}=\left\{  p\right\}  \times\lbrack0,\infty)$,
along with the natural choice of neighborhoods of infinity $(\left\langle
j\right\rangle \backslash X)\times(-\infty,-n]\cup\lbrack n,\infty)$ yield
representations of $\operatorname*{pro}$-$\pi_{1}\left(  \varepsilon\left(
(\left\langle j\right\rangle \backslash X)\times%
\mathbb{R}
\right)  ,r_{\pm}\right)  $ of the form $%
\mathbb{Z}
\overset{\operatorname*{id}}{\longleftarrow}%
\mathbb{Z}
\overset{\operatorname*{id}}{\longleftarrow}%
\mathbb{Z}
\overset{\operatorname*{id}}{\longleftarrow}\cdots$. The proper homotopy
equivalence promised above implies the same for the two ends of $X\times
\mathbb{S}^{1}$. Clearly, that can happen only if $X$ is simply connected at
each of its two ends.

In the non-cocompact case, $(\left\langle j\right\rangle \backslash X)\times%
\mathbb{R}
$ is 1-ended and by \cite[Prop. 3.12]{GM2}, with an appropriate choice of base
ray, $\operatorname*{pro}$-$\pi_{1}\left(  \varepsilon\left(  (\left\langle
j\right\rangle \backslash X)\times%
\mathbb{R}
\right)  ,r\right)  $ may be represented by an inverse sequence%
\begin{equation}
F_{0}\times\left\langle a\right\rangle \overset{\lambda_{1}\times
\operatorname*{id}}{\twoheadleftarrow}F_{1}\times\left\langle a\right\rangle
\overset{\lambda_{2}\times\operatorname*{id}}{\twoheadleftarrow}F_{2}%
\times\left\langle a\right\rangle \overset{\lambda_{3}\times\operatorname*{id}%
}{\twoheadleftarrow}\cdots\label{Inverse sequence: frees cross Z}%
\end{equation}
where each $F_{i}$ is a finitely generated free group, $\lambda_{i}$ takes
$F_{i+1}$ onto $F_{i}$, and $\left\langle a\right\rangle $ is an infinite
cyclic group corresponding to a `copy' of $\pi_{1}\left(  (\left\langle
j\right\rangle \backslash X)\times\left\{  r_{i}\right\}  \right)  $, for
increasingly large $r_{i}$. Semistability of this sequence implies that
$(\left\langle j\right\rangle \backslash X)\times%
\mathbb{R}
$, and hence $X\times\mathbb{S}^{1}$, is strongly connected at infinity. This
allows us to dispense with mention of base rays. It also implies that $X$ is
strongly connected at infinity, so $\operatorname*{pro}$-$\pi_{1}\left(
\varepsilon(X\right)  )$ is semistable and may be represented by an inverse
sequence of surjections $H_{0}\overset{\mu_{1}}{\twoheadleftarrow}%
H_{1}\overset{\mu_{2}}{\twoheadleftarrow}H_{2}\overset{\mu_{3}}%
{\twoheadleftarrow}\cdots$. It follows that $\operatorname*{pro}$-$\pi
_{1}(\varepsilon(X\times\mathbb{S}^{1}))$ may be represented by
\begin{equation}
H_{0}\times\left\langle t\right\rangle \overset{\mu_{1}\times
\operatorname*{id}}{\twoheadleftarrow}H_{1}\times\left\langle t\right\rangle
\overset{\mu_{2}\times\operatorname*{id}}{\twoheadleftarrow}H_{2}%
\times\left\langle t\right\rangle \overset{\mu_{3}\times\operatorname*{id}%
}{\twoheadleftarrow}\cdots\label{Inverse sequence: H-sequence cross Z}%
\end{equation}
where each $\left\langle t\right\rangle $ is the infinite cyclic group
corresponding to the $\mathbb{S}^{1}$-factor.

The equivalences of (\ref{Equivalences: Quotient x R versus X x S1}) ensure a
ladder diagram between subsequences of (\ref{Inverse sequence: frees cross Z})
and (\ref{Inverse sequence: H-sequence cross Z}). After relabeling to avoid
messy subsequence notation, that diagram has the form:
\begin{equation}
\begin{diagram} H_{0}\times\left\langle t\right\rangle & & \lTo^{\mu_{1}\times\operatorname*{id}} & & H_{1}\times\left\langle t\right\rangle & & \lTo^{\mu_{2}\times\operatorname*{id}} & & H_{2}\times\left\langle t\right\rangle & & \lTo^{\mu_{3}\times\operatorname*{id}} & & H_{3}\times\left\langle t\right\rangle& \cdots\\ & \luTo ^{u_{0}} & & \ldTo^{d_{1}} & & \luTo ^{u_{1}} & & \ldTo^{d_{2}} & & \luTo^{u_{2}} & & \ldTo^{d_{3}} &\\ & & F_{0}\times\left\langle a\right\rangle & & \lTo^{\lambda_{1}\times\operatorname*{id}} & & F_{1}\times\left\langle a\right\rangle & & \lTo^{\lambda_{2}\times\operatorname*{id}} & & F_{2}\times\left\langle a\right\rangle & & \lTo^ {\lambda_{3}\times\operatorname*{id}} & & \cdots \end{diagram} \label{Diagram: Big ladder}%
\end{equation}
A close look at the homeomorphism between $(\left\langle j\right\rangle
\backslash X)\times%
\mathbb{R}
$ and $\operatorname*{Tor}\nolimits_{j}\left(  X\right)  $, as described in
\cite[\S 8]{GG}, shows that, with appropriate choice of base rays, we may
arrange that each $u_{i}$ takes $a$ to $t$. Then, by commutativity, each
$d_{i}$ takes $t$ to $a$, each $u_{i}$ takes $F_{i}$ into $H_{i}$, and each
$d_{i}$ takes $H_{i}$ into $F_{i-1}$. So diagram (\ref{Diagram: Big ladder})
restricts to a diagram of the form
\begin{equation}
\begin{diagram} H_{0} & & \lTo^{\mu_{1}} & & H_{1} & & \lTo^{\mu_{2}} & & H_{2} & & \lTo^{\mu_{3}} & & H_{3} & \cdots\\ & \luTo & & \ldTo & & \luTo & & \ldTo & & \luTo & & \ldTo &\\ & & F_{0} & & \lTo^{\lambda_{1}} & & F_{1} & & \lTo^{\lambda_{2}} & & F_{2} & & \lTo^ {\lambda_{3}} & & \cdots \end{diagram} \label{Diagram: Small ladder}%
\end{equation}
which verifies that $\operatorname*{pro}$-$\pi_{1}\left(  \varepsilon\left(
X\right)  \right)  $ is pro-finitely generated free.
\end{proof}

\begin{theorem}
\label{Theorem: HCMs admitting z-actions}If a simply connected and finitely
dominated Hilbert cube manifold $X$ admits a proper $%
\mathbb{Z}
$-action generated by a homeomorphism $j:X\rightarrow X$ that is properly
homotopic to $\operatorname*{id}_{X}$, then $X$ is $\mathcal{Z}$-com\-pact\-ifi\-able.
\end{theorem}

\begin{proof}
If the action is not cocompact, the previous two lemmas together with Theorem
\ref{Theorem: Bass-Heller-Swan}, ensure that $X$ satisfies the conditions of
Theorem \ref{CS Theorem}. In the cocompact case, the same lemmas imply that
$X$ is inward tame and 2-ended, and that each of those ends is simply
connected. In order to use the 1-ended version of Theorem \ref{CS Theorem}
provided here, split $X$ into a pair of 1-ended Hilbert cube manifolds and
apply the theorem to each end individually.
\end{proof}

\begin{theorem}
If a simply connected, locally finite polyhedron $P$ is finitely dominated and
finite-dimensional, and admits a proper $%
\mathbb{Z}
$-action generated by a homeomorphism $j:P\rightarrow P$ that is properly
homotopic to $\operatorname*{id}_{P}$, then $P\times\mathbb{I}^{2\cdot\dim
P+5}$ is $\mathcal{Z}$-com\-pact\-ifi\-able.
\end{theorem}

\begin{proof}
By Theorem \ref{Theorem: Edwards HCM Theorem}, $j\times\operatorname*{id}%
_{\mathcal{Q}}:P\times\mathcal{Q\ }\mathcal{\rightarrow}P\times\mathcal{Q}$
satisfies the hypotheses of Theorem \ref{Theorem: HCMs admitting z-actions}.
An application of Theorem \ref{Theorem: Ferry's stabilization theorem}
completes the proof.
\end{proof}

\subsection{Spaces that are simply connected at infinity}

The key result about Hilbert cube manifolds that are simply connected at
infinity is our easiest application of Theorem \ref{CS Theorem}; it can be
found in Chapman and Siebenmann's original work. For completeness, we include
a sketch of their proof.

\begin{theorem}
[{\cite[Cor. to Th.8]{CS}}]%
\label{Theorem: 1-conn at infinity Z-compactifiablility theorem}If $X$ is a
Hilbert cube manifold that is simply connected at infinity and $H_{\ast
}\left(  X;%
\mathbb{Z}
\right)  $ is finitely generated, then $X$ is $\mathcal{Z}$-com\-pact\-ifi\-able.

\begin{proof}
[Sketch of Proof]Due to the triviality of $\operatorname*{pro}$-$\pi
_{1}\left(  \varepsilon\left(  X\right)  \right)  $, we need only show that
$X$ is inward tame. If $N$ is a clean neighborhood of infinity, then
$\operatorname*{Bd}_{X}N$ is homotopy equivalent to a finite complex. Since
$H_{i}\left(  \operatorname*{Bd}\nolimits_{X}N;%
\mathbb{Z}
\right)  $ and $H_{i}\left(  X;%
\mathbb{Z}
\right)  $ are both finitely generated for all $i$, and eventually trivial,
the Mayer-Vietoris sequence%

\[
\cdots\rightarrow H_{i}\left(  \operatorname*{Bd}\nolimits_{X}N;%
\mathbb{Z}
\right)  \rightarrow H_{i}\left(  \overline{X-N};%
\mathbb{Z}
\right)  \oplus H_{i}\left(  N;%
\mathbb{Z}
\right)  \rightarrow H_{i}\left(  X;%
\mathbb{Z}
\right)  \rightarrow\cdots
\]
shows that the same is true for $H_{i}\left(  N;%
\mathbb{Z}
\right)  $. Furthermore, the simple connectivity at infinity of $X$, together
with standard techniques from Hilbert cube manifold topology, ensure the
existence of arbitrarily small simply connected $N$. Since a simply connected
complex with finitely generated $%
\mathbb{Z}
$-homology necessarily has finite homotopy type (see \cite[p.420]{Sp}), it
follows that $X$ is inward tame.
\end{proof}
\end{theorem}

\begin{theorem}
If $P$ is a finite-dimensional, locally finite polyhedron that is simply
connected at infinity and $H_{\ast}\left(  P;%
\mathbb{Z}
\right)  $ is finitely generated, then $P\times\mathbb{I}^{2\cdot\dim P+5}$ is
$\mathcal{Z}$-com\-pact\-ifi\-able.
\end{theorem}

\begin{proof}
Apply Theorems \ref{Theorem: Edwards HCM Theorem},
\ref{Theorem: 1-conn at infinity Z-compactifiablility theorem},
\ref{Theorem: Ferry's stabilization theorem}.
\end{proof}

\section{Proofs of the main theorems\label{Section: Proofs of main theorems}}

We now provide proofs of the unverified theorems from
\S \ref{Section: Introduction}. Theorems \ref{Theorem 3} and \ref{Theorem 4}
require only an assemby of ingredients from \S 2 and \S 3, so we begin there.

\begin{proof}
[Proof of Theorem \ref{Theorem 3}]Since $\widetilde{K}$ is contractible, both
$X$ and $Y$ are also contractible. By Lemmas
\ref{Lemma: inward tameness of products} and
\ref{Lemma: products that are fg free at infinity}, $X\times Y$ is inward tame
and 1-ended with $\operatorname*{pro}$-$\pi_{1}\left(  X\times Y,r\right)  $
that is pro-finitely generated free, and since $\widetilde{K}\overset
{p}{\simeq}X\times Y$, each of these properties is inherited by $\widetilde
{K}$. Applying Theorems \ref{Theorem: Edwards HCM Theorem},
\ref{Theorem: Bass-Heller-Swan}, and \ref{CS Theorem} in the usual way
provides a $\mathcal{Z}$-com\-pact\-ific\-at\-ion of $\widetilde{K}%
\times\mathcal{Q}$, and since $\dim\widetilde{K}=\dim K<\infty$, Theorem
\ref{Theorem: Ferry's stabilization theorem} provides a $\mathcal{Z}%
$-com\-pact\-ific\-at\-ion of the ER $\widetilde{K}\times\mathbb{I}%
^{2\cdot\dim K+5}$.
\end{proof}

\begin{proof}
[Proof of Theorem \ref{Theorem 4}]By Lemmas
\ref{Lemma: inward tameness of spaces admitting nice Z-actions} and
\ref{Lemma: pro-pi1 of spaces admitting Z-actions}, $X$ is inward tame, and
either: 2-ended and simply connected at each end; or 1-ended with pro-finitely
generated free fundamental group at infinity. By proper homotopy invariance,
the same is true for $\widetilde{K}$, so by the usual argument, $\widetilde
{K}\times\mathcal{Q}$ is $\mathcal{Z}$-com\-pact\-ifi\-able. Another
application of Theorem \ref{Theorem: Ferry's stabilization theorem} provides a
$\mathcal{Z}$-com\-pact\-ific\-at\-ion of $\widetilde{K}\times\mathbb{I}%
^{2\cdot\dim K+5}$.
\end{proof}

\begin{remark}
\emph{In the special case, where }$X$\emph{ (or }$\widetilde{K}$\emph{) admits
a cocompact }$%
\mathbb{Z}
$\emph{-action, the above argument is overkill. There, since }$X$\emph{ is
contractible, }$\left\langle j\right\rangle \backslash X$\emph{ is homotopy
equivalent to a circle; and since }$\left\langle j\right\rangle \backslash
X$\emph{ is compact, a homotopy equivalence }$f:\left\langle j\right\rangle
\backslash X\rightarrow S^{1}$\emph{ lifts to a proper homotopy equivalence
}$X\overset{p}{\simeq}%
\mathbb{R}
$\emph{. It is then straightforward to show that the 2-point
com\-pact\-ifi\-ca\-tions of }$X$\emph{ and }$\widetilde{K}$\emph{ are
themselves }$\mathcal{Z}$\emph{-com\-pact\-ifi\-ca\-tions.}
\end{remark}

To obtain the full strength of Theorem \ref{Theorem 5}, we require a new
ingredient from \cite{GG}.

\begin{proof}
[Proof of Theorem \ref{Theorem 5}]Since $G$ is type $F$, each nontrivial
element has infinite order; so we may apply \cite[Th.1.4]{GG} to conclude that
$G$ is either simply connected at infinity or $G$ is virtually a surface
group. In other words, if $K$ is a finite K($G,1$) complex, then
$\widetilde{K}$ is either simply connected at infinity, or $\widetilde{K}$ is
the universal cover of the corresponding finite K($H,1$) complex
$H\backslash\widetilde{K}$, where $H$ is a finite index subgroup of $G$ and
$H\cong\pi_{1}\left(  S\right)  $, where $S$ is a closed surface with infinite
fundamental group. (\textbf{Note.} By \cite{CJ} or \cite{Ga} a torsion-free
virtual surface group is, in fact, a surface group; but that fact is not
needed here.)

In the case where $\widetilde{K}$ is simply connected at infinity, we may
apply Theorem \ref{Theorem: 1-conn at infinity Z-compactifiablility theorem}
to conclude that $\widetilde{K}\times\mathcal{Q}$ is $\mathcal{Z}%
$-com\-pact\-ifi\-able, and hence $\widetilde{K}\times\mathbb{I}^{2\cdot\dim
K+5}$, admits the desired $\mathcal{Z}$-com\-pact\-ific\-at\-ion.

In the second case, we may conclude that $\widetilde{K}\overset{p}{\simeq
}\widetilde{S}\approx\mathbb{%
\mathbb{R}
}^{2}$. It follows that $\widetilde{K}$ is 1-ended and inward tame, with
$\operatorname*{pro}$-$\pi_{1}\left(  \varepsilon\left(  \widetilde{K}\right)
\right)  $ stably isomorphic to $%
\mathbb{Z}
$. By Theorem \ref{Theorem: Bass-Heller-Swan}, $\widetilde{K}\times
\mathcal{Q}$ satisfies the hypotheses of Theorem \ref{CS Theorem}, and is
therefore $\mathcal{Z}$-com\-pact\-ifi\-able. Another application of Theorem
\ref{Theorem: Ferry's stabilization theorem} completes the proof.\medskip
\end{proof}

Theorems \ref{Theorem 2} is a special case of Theorem \ref{Theorem 4}, so a
proof of Theorem \ref{Theorem 1} is all that remains. It is a consequence of
Theorem \ref{Theorem 3} and the following crucial lemma.

\begin{lemma}
\label{Lemma: properties of classifying spaces for group extensions}Let
$1\rightarrow N\rightarrow G\rightarrow Q\rightarrow1$ be a short exact
sequence of groups where both $N$ and $Q$ have type F, then $G$ also has type
F. Moreover, if $Y$ and $Z$ are finite classifying spaces for $N$ and $Q$,
respectively, then $G$ admits a finite classifying space $W^{\prime}$ with the
property that $\widetilde{W}^{\prime}$ is proper homotopy equivalent to
$\widetilde{Y}\times\widetilde{Z}$.
\end{lemma}

The first sentence of this lemma follows from techniques laid out in \S 6.1,
\S 7.1, and \S 7.2 of \cite{Ge2}; the second sentence is essentially a
restatement of Proposition 17.3.1 of \cite{Ge2}. Due to their importance in
this paper, we provide a guide to those arguments in the following outline.
Since Proposition 17.3.1 in \cite{Ge2} is light on details (and since we had
worked out an alternative approach prior to discovering that proposition), we
have included an appendix with an alternative proof. A novel aspect of the
approach presented there is its use of \textquotedblleft approximate
fibrations\textquotedblright.

\begin{proof}
[Proof of Lemma
\ref{Lemma: properties of classifying spaces for group extensions}
(outline)]The construction of a finite K($G,1$) complex is obtained by an
application of the Borel construction followed by the Rebuilding Lemma (see
\cite[\S 6.1]{Ge2}). For the Borel construction, begin with a (not necessarily
finite) K($G,1$) complex $X$ and let $G$ act diagonally on $\widetilde
{X}\times\widetilde{Z}$ , where the (nonfree) action of $G$ on $\widetilde{Z}$
is induced by the quotient map $G\rightarrow Q$. Since the diagonal action
itself is free, the quotient $W=G\backslash\left(  \widetilde{X}%
\times\widetilde{Z}\right)  $ is a K($G,1$) complex---probably not finite.
Inspection of this quotient space reveals a natural projection map
$q:W\rightarrow Z$ that is a fiber bundle with fiber the aspherical CW complex
$N\backslash\widetilde{X}$.

Next is the rebuilding\ stage of the argument. Here the K($G,1$) complex $W$
is \textquotedblleft rebuilt\textquotedblright\ by replacing each fiber
$N\backslash\widetilde{X}$ of the map $q:W\rightarrow Z$ with the homotopy
equivalent (but finite) complex $Y$. This is done inductively over the skeleta
of $Z$: first a copy of $Y$ is placed over each vertex of $Z$, then over each
edge $e$ of $Z$ a copy of $Y\times\left[  0,1\right]  $ is attached with
$Y\times0$ being glued to the copy of $Y$ lying over the initial vertex of $e$
and $Y\times1$ glued to the copy of $Y$ lying over the terminal vertex of $e$.
From there we move to the 2-cells of $Z$, and so on. At each step, the bundle
map $q$ provides instructions for the gluing maps. At the end we have a
bundle-like \textquotedblleft stack\textquotedblright\ of CW complexes
$q^{\prime}:W^{\prime}\rightarrow Z$ with each point preimage a copy of $Y$
and a homotopy equivalence $k:W^{\prime}\rightarrow W$. Since both $Z$ and $Y$
are finite complexes, $W^{\prime}$ is a finite complex, so $G$ has type F.

Obtaining a proper homotopy equivalence $h:\widetilde{W}^{\prime}%
\rightarrow\widetilde{Y}\times\widetilde{Z}$ is an interesting and delicate
task. A proof can be found in the appendix; otherwise, the reader is referred
to \cite[Prop.17.3.1]{Ge2}.
\end{proof}

\appendix

\section{An alternative approach to Lemma
\ref{Lemma: properties of classifying spaces for group extensions}}

In this appendix we take a closer look at the proper homotopy equivalence
promised in Lemma
\ref{Lemma: properties of classifying spaces for group extensions} and offer
an alternative to the proof suggested in \cite{Ge2}. Begin with a short exact
sequence of groups $1\rightarrow N\rightarrow G\rightarrow Q\rightarrow1$
where both $N$ and $Q$ have type F. Then, as described in the sketched proof
of Lemma \ref{Lemma: properties of classifying spaces for group extensions},
$G$ also has type F. More specifically, if $Y$ is a finite K($N,1$) complex
and $Z$ is a finite K($Q,1$) complex, then there is a finite K($G,1$) complex
$W^{\prime}$, obtained by an application of the Borel construction followed by
the Rebuilding Lemma. As a corollary of the construction, $W^{\prime}$ comes
equipped with a map $q:W^{\prime}\rightarrow Z$ for which each point preimage
is a copy of $Y$. In fact, for each open $k$-cell $\mathring{e}^{k}$ of $Z$,
$q^{-1}\left(  \mathring{e}^{k}\right)  \approx\mathring{e}^{k}\times Y$.

Although $q:W^{\prime}\rightarrow Z$ is not necessarily a fiber bundle, it
exhibits many properties of a fiber bundle; it is a stack of CW complexes over
$Z$ with fiber $Y$. If we let $\widehat{W}^{\prime}$ be the intermediate cover
of $W^{\prime}$ corresponding to $N\trianglelefteq G$, we get another stack of
CW complexes $\widehat{q}:\widehat{W}^{\prime}\rightarrow\widetilde{Z}$ over
the contractible space $\widetilde{Z}$. Given the standard fact that a fiber
bundle over a contractible space is always a product bundle, it is reasonable
to hope that, in the case at hand, $\widehat{W}^{\prime}$ is \textquotedblleft
approximately a product\textquotedblright. By using the aptly named theory of
\textquotedblleft approximate fibrations\textquotedblright, we will eventually
arrive at the following main result of this appendix.

\begin{proposition}
\label{Prop: Main result of Appendix}Given the above setup, $\widehat
{W}^{\prime}$ is proper homotopy equivalent to $Y\times\widetilde{Z}$.
\end{proposition}

This result is stronger than needed to complete Lemma
\ref{Lemma: properties of classifying spaces for group extensions}, and may be
of interest in its own right. Lemma
\ref{Lemma: properties of classifying spaces for group extensions} is obtained
from Proposition \ref{Prop: Main result of Appendix} by lifting the promised
proper homotopy equivalence to the universal covers. It is worth noting that
$W^{\prime}$, itself, is typically not homotopy equivalent to $Y\times Z$.

In this appendix, we first provide a constructive\ proof of the special case
where $Q$ is infinite cyclic; in that case $G$ is a semidirect product $G\cong
N\rtimes_{\varphi}%
\mathbb{Z}
$. The special case motivates the work to be done later and also illustrates
the subtleties that are overcome with the general theory. After completing the
special case, we will provide a brief overview of the theory of
\textquotedblleft approximate fibrations\textquotedblright. Then we employ
that theory to prove Proposition \ref{Prop: Main result of Appendix} in full generality.

\subsection{Mapping tori of self-homotopy
equivalences\label{Subsection: Mapping tori of self-homotopy equivalences}}

In this section we focus on the special case of Proposition
\ref{Prop: Main result of Appendix}, where $G$ is an extension of the form
$1\rightarrow N\rightarrow G\rightarrow%
\mathbb{Z}
\rightarrow1$; equivalently, $G\cong N\rtimes_{\varphi}%
\mathbb{Z}
$ for some automorphism $\varphi:G\rightarrow G$. In this case, the
Borel/Rebuilding procedure yields a finite K($G$,1) complex that is the
mapping torus of a map $f:Y\rightarrow Y$, with $f_{\#}=\varphi$. Since $Y$ is
a K($H$,1), $f$ is necessarily a homotopy equivalence. The goal of this
section then becomes:

\begin{lemma}
\label{Lemma: Universal cover of a mapping torus}If $K$ is a compact connected
ANR and $f:K\rightarrow K$ is homotopy equivalence, then the canonical
infinite cyclic cover, $\operatorname*{Tel}\nolimits_{f}\left(  K\right)  $,
of $\operatorname*{Tor}\nolimits_{f}\left(  K\right)  $ is proper homotopy
equivalent to $K\times%
\mathbb{R}
$.
\end{lemma}

\begin{proof}
Let $g:K\rightarrow K$ be a homotopy inverse for $f$ and $B:K\times\left[
0,1\right]  \rightarrow K$ with $B_{0}=\operatorname*{id}_{K}$ and $B_{1}=fg$.
In accordance with Lemma
\ref{Lemma: mapping cylinders of homotopy equivalences}, our goal is to define
a map $u:K\times%
\mathbb{R}
\rightarrow\operatorname*{Tel}\nolimits_{f}\left(  K\right)  $, such that
there is a proper strong deformation retraction of $\mathcal{M}_{\left[
0,1\right]  }\left(  u\right)  $ onto the domain copy of $K\times%
\mathbb{R}
$. For each integer $n$, define a function $u_{n}:K\times\left[  n,n+1\right]
\rightarrow\mathcal{M}_{\left[  n,n+1\right]  }\left(  f\right)  $ by the
rule:\smallskip%
\[
u_{n}\left(  x,r\right)  =q_{\left[  n,n+1\right]  }(B_{r-n}\left(
g^{n}\left(  x\right)  \right)  ,r)\text{, when }n\geq0
\]
and%
\[
u_{n}\left(  x,r\right)  =q_{\left[  n,n+1\right]  }\left(  f^{-n}\left(
x),r\right)  \right)  \text{, when }n<0\text{.\smallskip}%
\]
Here it is understood that $g^{0}=\operatorname*{id}_{K}$.

Note that
\[
u_{-1}\left(  x,0\right)  \allowbreak=\allowbreak q_{[-1,0]}\left(  f\left(
x),0\right)  \right)  \text{, while }u_{0}\left(  x,0\right)  \allowbreak
=\allowbreak q_{\left[  0,1\right]  }\left(  B_{0}(x),0\right)  \allowbreak
=\allowbreak q_{\left[  0,1\right]  }\left(  x,0\right)
\]
and for each integer $n\geq1$,
\[
u_{n-1}\left(  x,n\right)  \allowbreak=\allowbreak q_{[n-1,n]}\left(
B_{1}\left(  g^{n-1}\left(  x\right)  \right)  \right)  \allowbreak
=\allowbreak q_{[n-1,n]}\left(  fgg^{n-1}\left(  x\right)  ,n\right)
\allowbreak=\allowbreak q_{[n-1,n]}\left(  fg^{n}\left(  x\right)  ,n\right)
\]
and
\[
u_{n}\left(  x,n\right)  \allowbreak=\allowbreak q_{[n,n+1]}\left(
B_{0}\left(  g^{n}\left(  x\right)  \right)  ,n\right)  \allowbreak
=\allowbreak q_{[n,n+1]}\left(  g^{n}\left(  x\right)  ,n\right)  .
\]
Similarly, for each for each integer $n\leq-1$,
\[
u_{n-1}\left(  x,n\right)  \allowbreak=\allowbreak q_{[n-1,n]}\left(
f^{-(n-1)}\left(  x),n\right)  \right)  \allowbreak=\allowbreak q_{[n-1,n]}%
\left(  f(f^{-n}\left(  x),n\right)  \right)
\]
and
\[
u_{n}\left(  x,n\right)  \allowbreak=\allowbreak q_{[n,n+1]}\left(
f^{-n}\left(  x),n\right)  \right)  \text{.}%
\]
It follows that the $u_{n}$ can be glued together to obtain a proper map
$u:K\times%
\mathbb{R}
\rightarrow\operatorname*{Tel}\nolimits_{f}\left(  K\right)  $. See Figure
\ref{Fig 1}.\medskip%
\begin{figure}
[ptb]
\begin{center}
\includegraphics[
height=1.74in,
width=5.5486in
]%
{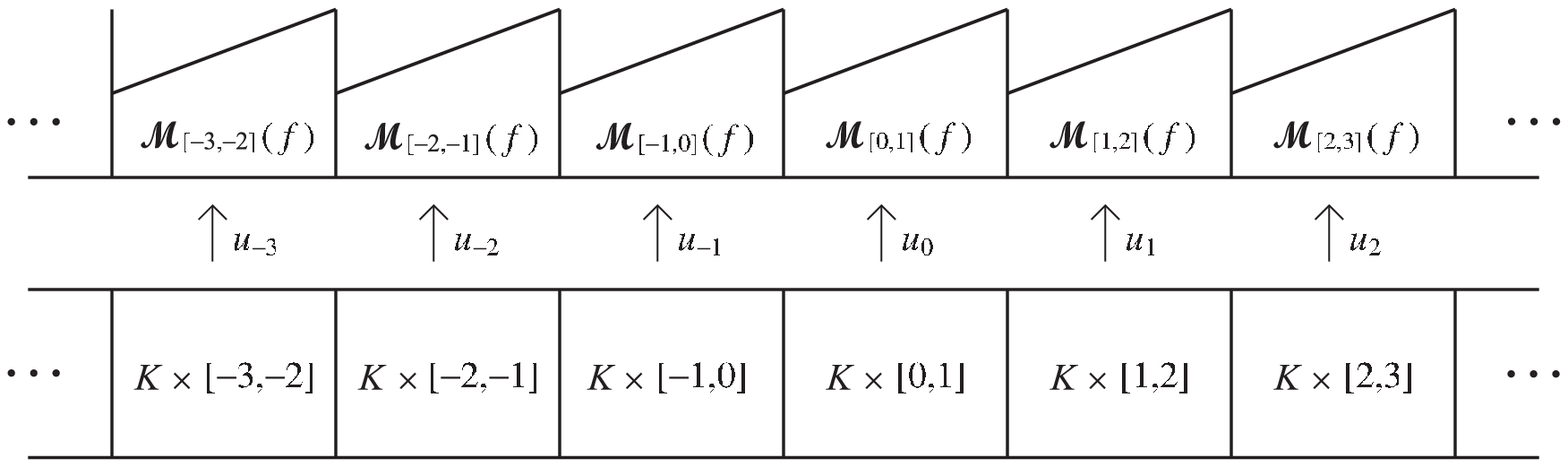}%
\caption{The map $u:K\times\mathbb{R} \rightarrow\operatorname{Tel}_{f}\left(
K\right)  $.}%
\label{Fig 1}%
\end{center}
\end{figure}
\smallskip

\noindent\textsc{Claim.}\textbf{ }\emph{There is a proper strong deformation
retraction of} $\mathcal{M}_{\left[  0,1\right]  }\left(  u\right)  $
\emph{onto} $K\times%
\mathbb{R}
$.\medskip

First note that, since $u$ respects $%
\mathbb{R}
$-coordinates, the natural projections $K\times%
\mathbb{R}
\rightarrow%
\mathbb{R}
$ and $p:\operatorname*{Tel}\nolimits_{f}\left(  K\right)  \rightarrow%
\mathbb{R}
$ can be extended to a projection $\widehat{p}:\mathcal{M}_{\left[
0,1\right]  }\left(  u\right)  \rightarrow%
\mathbb{R}
$ with the property that each point preimage $\widehat{p}\,^{-1}\left(
r\right)  $ is a mapping cylinder $C_{r}$ of a map from $K\times\left\{
r\right\}  $ to $K_{r}$. Indeed, for an integer $n\geq0$, $C_{n}$ is the
mapping cylinder of $fg^{n}$ and for an integer $n<0$, $C_{n}$ is the mapping
cylinder of $f^{-\left(  n-1\right)  }$. So each $C_{n}$ is a mapping cylinder
of a homotopy equivalence---a fact that will be useful later. (In fact, each
$C_{r}$ is a mapping cylinder of a homotopy equivalence, but this fact will
only be used for integral values of $r$.) Note also that $\mathcal{M}_{\left[
0,1\right]  }\left(  u\right)  $ may be viewed as a countable union
$\bigcup_{n\in%
\mathbb{Z}
}\mathcal{M}_{[0,1]}\left(  u_{n}\right)  $, where each $\mathcal{M}%
_{[0,1]}\left(  u_{n}\right)  $ intersects $\mathcal{M}_{[0,1]}\left(
u_{n-1}\right)  $ in $C_{n}$.\medskip

\noindent\textsc{Subclaim.}\textbf{ }\emph{For each} $n$, $\mathcal{M}%
_{[0,1]}\left(  u_{n}\right)  $ \emph{strong deformation retracts onto the
subset} \linebreak$C_{n}\cup\left(  K\times\left[  n,n+1\right]  \right)
_{1}\cup C_{n+1}$.\medskip

It suffices to show that $C_{n}\cup\left(  K\times\left[  n,n+1\right]
\right)  _{1}\cup C_{n+1}\hookrightarrow\mathcal{M}\left(  u_{n}\right)  $ is
a homotopy equivalence. Since $C_{n}$ and $C_{n+1}$ are mapping cylinders of
homotopy equivalences, each strong deformation retracts onto its domain end,
so $\left(  K\times\left[  n,n+1\right]  \right)  _{1}\hookrightarrow
C_{n}\cup\left(  K\times\left[  n,n+1\right]  \right)  _{1}\cup C_{n+1}$ is a
homotopy equivalence; therefore, it is enough to show that $\left(
K\times\left[  n,n+1\right]  \right)  _{1}\hookrightarrow\mathcal{M}_{\left[
0,1\right]  }\left(  u_{n}\right)  $ is a homotopy equivalence. Note that the
inclusions $K_{n}\hookrightarrow C_{n}$, $K_{n}\hookrightarrow\mathcal{M}%
_{\left[  n,n+1\right]  }\left(  f\right)  $ and $\mathcal{M}_{\left[
n,n+1\right]  }\left(  f\right)  \hookrightarrow\mathcal{M}_{\left[
0,1\right]  }\left(  u_{n}\right)  $ are all homotopy equivalences, since each
subspace is the range end of a corresponding mapping cylinder. It follows that
$C_{n}\hookrightarrow\mathcal{M}_{\left[  0,1\right]  }\left(  u_{n}\right)  $
is a homotopy equivalence, and since $K\times\left\{  n\right\}
\hookrightarrow C_{n}$ is a homotopy equivalence it follows that
$K\times\left\{  n\right\}  \hookrightarrow\mathcal{M}_{\left[  0,1\right]
}\left(  u_{n}\right)  $, and hence, $\left(  K\times\left[  n,n+1\right]
\right)  _{1}\hookrightarrow\mathcal{M}_{\left[  0,1\right]  }\left(
u_{n}\right)  $ is a homotopy equivalence. The subclaim follows.\smallskip

To complete the claim, first properly strong deformation retract
$\mathcal{M}_{\left[  0,1\right]  }\left(  u\right)  $ onto $(K\times%
\mathbb{R}
)_{1}\cup\left(  \bigcup\nolimits_{n\in%
\mathbb{Z}
}C_{n}\right)  $ using the union of the strong deformation retractions
provided by the subclaim. Follow that by a strong deformation of $(K\times%
\mathbb{R}
)_{1}\cup\left(  \bigcup\nolimits_{n\in%
\mathbb{Z}
}C_{n}\right)  $ onto $\left(  K\times%
\mathbb{R}
\right)  _{1}$ obtained by individually strong deformation retracting each
$C_{n}$ onto its domain end.
\end{proof}

\begin{remark}
\emph{The delicate nature of defining }$u:K\times%
\mathbb{R}
\rightarrow\operatorname*{Tel}\nolimits_{f}\left(  K\right)  $\emph{, in the
above proof, hints at the subtelty of Lemma
\ref{Lemma: properties of classifying spaces for group extensions}.}
\end{remark}

\subsection{Approximate fibrations}

We now review the main definitions and a few fundmental facts from the theory
of approximate fibrations---a theory developed by Coram and Duvall
\cite{CD1},\cite{CD2} to generalize the notions of fibration and fiber bundle.

A proper surjective map $p:E\rightarrow B$ between (locally compact, metric)
ANRs is an \emph{approximate fibration} if it satisfies the following
\emph{approximate lifting property}:\bigskip

\noindent\emph{For every homotopy }$H:X\times\left[  0,1\right]  \rightarrow
B$\emph{, map }$h:X\rightarrow E$\emph{ with }$\pi h=H_{0}$\emph{, and open
cover }$\mathcal{U}$\emph{ of }$B$\emph{, there exists }$\overline{H}%
:X\times\left[  0,1\right]  \rightarrow E$\emph{ such that }$\overline{H}%
_{0}=h$\emph{ and }$p\overline{H}$\emph{ is }$U$\emph{-close to }$H$\emph{
(that is, for each }$\left(  x,t\right)  \in X\times\left[  0,1\right]
$\emph{ there exists }$U\in\mathcal{U}$\emph{ containing both }$H\left(
x,t\right)  $\emph{ and }$p\overline{H}\left(  x,t\right)  $\emph{).\bigskip}

\noindent For each $b\in B$, $F_{b}:=$ $p^{-1}\left(  b\right)  $ is called a
\emph{fiber}. Approximate fibrations allows for fibers with bad local
properties; however, the theory is easier, but still rich, when fibers are
ANRs. Since fibers of the maps considered in this paper are always ANRs (in
fact, finite CW complexes), we will focus on that special case. In this
context, there is a particularly nice criterion for recognizing an approximate fibration.

Suppose $p:E\rightarrow B$ is a proper map between connected ANRs with ANR
fibers. Then, for each fiber $F_{b}$, some neighborhood $U_{b}$ retracts onto
$F_{b}$, and for points $b^{\prime}$ sufficiently close to $b$, this induces a
map of $F_{b^{\prime}}$ to $F_{b}$. By \cite{CD2}, $p:E\rightarrow B$ is an
approximate fibration if and only if each $b\in B$ has a neighborhood over
which each of these induced fiber maps is a homotopy equivalence.

\begin{example}
If $f:K\rightarrow K$ is homotopy equivalence of a compact connected ANR to
itself, the above criterion is easily applied to show that quotient maps
$p_{1}:\mathcal{M}_{\left[  a,b\right]  }\left(  f\right)  \rightarrow\left[
a,b\right]  $, $p_{2}:\operatorname*{Tor}\nolimits_{f}\left(  K\right)
\rightarrow\mathbb{S}^{1}$, and $p_{3}:\operatorname*{Tel}\nolimits_{f}\left(
K\right)  \rightarrow%
\mathbb{R}
$ are all approximate fibrations. In the simple case where $K$ is an arc and
$f$ is a constant map, these projections are not actual fibrations. In the
case where $K$ is a bouquet of circles and $f$ is an arbitrary map inducing a
$\pi_{1}$-isomorphism, the examples are already of group theoretic interest.
\end{example}

As with the case of ordinary fibrations, approximate fibrations give rise to
homotopy long exact sequences \cite{CD1}. When no restrictions are placed on
the fibers, these sequence involve the shape (or \v{C}ech) homotopy groups of
the fiber; when the fibers are ANRs that technicality vanishes and we have:

\begin{lemma}
Let $p:E\rightarrow B$ be an approximate fibration between connected ANRs with
connected ANR fibers. Then, for any $b\in B$, there is a long exact sequence
\[
\cdots\rightarrow\pi_{k+1}\left(  B\right)  \rightarrow\pi_{k}\left(
F_{b}\right)  \overset{i_{\#}}{\rightarrow}\pi_{k}\left(  E\right)
\overset{p_{\#}}{\rightarrow}\pi_{k}\left(  B\right)  \rightarrow\pi
_{k-1}\left(  F_{b}\right)  \rightarrow\cdots
\]
where $i$ is the inclusion map.
\end{lemma}

We now prove a general fact about approximations that is almost tailor-made
for proving Proposition \ref{Prop: Main result of Appendix}.

\begin{theorem}
\label{Theorem: approx. fibrations over contractible base}Let $p:E\rightarrow
B$ be an approximate fibration between connected ANRs with connected ANR
fibers and let $b\in B$. If $B$ is contractible then $E$ is proper homotopy
equivalent to $F_{b}\times B$.
\end{theorem}

\begin{proof}
By the homotopy long exact sequence and an application of the Whitehead
Theorem, $F_{b}\overset{i}{\hookrightarrow}E$ is a homotopy equivalence. Let
$r_{b}:E\rightarrow F_{b}$ be a homotopy inverse that retracts $E$ onto
$F_{b}$. We will observe that $r_{b}\times p:E\rightarrow F_{b}\times B$ is a
proper homotopy equivalence.

Clearly $r_{b}\times p$ is proper, and by contractibility of $B$, it is a
homotopy equivalence. We will show that it is a proper homotopy equivalence by
exhibiting cofinal sequences of neighborhoods of infinity in the domain and
range, respectively, such that $r_{b}\times p$ restricts to homotopy
equivalences between corresponding entries. An application of the Proper
Whitehead Theorem \cite[Th.17.1.1]{Ge2} or \cite[Prop.IV]{Si} (applied to the
mapping cylinder of $r_{b}\times p$) completes the proof.

Let $\left\{  V_{i}\right\}  _{i=0}^{\infty}$ be a cofinal nested sequence of
neighborhoods of infinity in $B$. For convenience, assume that $B$ is 1-ended
and that each $V_{i}$ is chosen to be connected; we will return to the general
case momentarily. For each $i$, let $U_{i}=p^{-1}\left(  V_{i}\right)  $. Then
$\left\{  U_{i}\right\}  $ and $\left\{  F_{b}\times V_{i}\right\}  $ are
nested cofinal sequence of neighborhoods of infinity in $E$ and $F_{b}\times
B$, respectively. Morevover, $r_{b}\times p$ retricts to a map of $U_{i}$ to
$F_{b}\times V_{i}$, for each $i$.

Note that the restriction $p_{i}:U_{i}\rightarrow V_{i}$ is, itself, an
approximate fibration. By choosing $b$ to lie in $V_{i}$, and recalling that
the composition $F_{b}\hookrightarrow U_{i}\hookrightarrow E$ induces $\pi
_{k}$-isomorphisms, for all $k$, we see that the long exact sequence for
$p_{i}:U_{i}\rightarrow V_{i}$ yields a short exact sequence%
\[
1\rightarrow\pi_{k}\left(  F_{b}\right)  \overset{i_{\#}}{\rightarrow}\pi
_{k}\left(  U_{i}\right)  \overset{p_{i\#}}{\rightarrow}\pi_{k}\left(
V_{i}\right)  \rightarrow1
\]
for each $k$. Since $U_{i}$ retracts onto $F_{b}$, these sequences split; so
$\pi_{k}\left(  U_{i}\right)  \cong\pi_{k}\left(  F_{b}\right)  \times\pi
_{k}\left(  V_{i}\right)  $. From there it is easy to see that each
restriction of $r_{b}\times p$ induces isomorphisms $\pi_{k}\left(
U_{i}\right)  \rightarrow\pi_{k}\left(  F_{b}\times V_{i}\right)  $,
completing \ the proof.

If $B$ has more than one end, one simply applies the above argument to
individual components of the $V_{i}$.
\end{proof}

\begin{proof}
[Proof of Proposition \ref{Prop: Main result of Appendix}]By Theorem
\ref{Theorem: approx. fibrations over contractible base}, it suffices to show
that the stack of CW complexes $\widehat{q}:\widehat{W}^{\prime}%
\rightarrow\widetilde{Z}$ is an approximate fibration. Since each point
preimage is a copy of the finite complex $Y$, we need only check that these
fibers \textquotedblleft line up homotopically\textquotedblright\ in the sense
of the approximate fibration recognition criterion described above.

From the initial Borel/Rebuilding construction, it is clear that the inclusion
of each fiber into $\widehat{W}^{\prime}$ indicies a $\pi_{1}$-isomorphism;
and, since both $Y$ and $\widehat{W}^{\prime}$ are aspherical, these
inclusions are homotopy equivalences. So any retraction of $\widehat
{W}^{\prime}$ onto a fiber $Y$ restricts to homotopy equivalences between
fibers. The recognition criterion follows.
\end{proof}

\end{document}